\documentclass[11pt]{amsart}
\usepackage{amssymb,amsmath,amsthm,newlfont,enumerate, standalone}
\usepackage{tikz-cd}
\usepackage{pgfplots}
\pgfplotsset{compat=1.18} 
\usepgfplotslibrary{fillbetween}
\usetikzlibrary{intersections}
\usepackage{mathtools}
\usetikzlibrary{patterns}
\theoremstyle{plain}
\newtheorem{theorem}{Theorem}
\newtheorem{lemma}[theorem]{Lemma}
\newtheorem{proposition}[theorem]{Proposition}
\newtheorem{corollary}[theorem]{Corollary}

\newtheorem*{question*}{Question}

\theoremstyle{definition}

\theoremstyle{remark}

\newcommand{\CC}{\mathbb{C}}
\newcommand{\DD}{\mathbb{D}}
\newcommand{\RR}{\mathbb{R}}

\newcommand{\A}{\mathcal {A}}
\newcommand{\NN}{\mathcal {N}}
\newcommand{\LL}{\mathcal {L}}

\newcommand{\HH}{\mathcal{H}}

\let\Re\undefined
\let\Im\undefined
\DeclareMathOperator{\Re}{\mathrm{Re}}
\DeclareMathOperator{\Im}{\mathrm{Im}}

\begin{document}

\title[Double-layer potentials, configuration constants and applications]{Double-layer potentials, configuration constants and applications to numerical ranges}

\author{Bartosz Malman}
\author{Javad Mashreghi}
\author{Ryan O'Loughlin}
\author{Thomas Ransford}

\address{Division of Mathematics and Physics, 
        Mälardalen University,
		Västerås, Sweden}
\email{bartosz.malman@mdu.se} 

\address{Département de mathématiques et de statistique, Université Laval, Québec (QC), G1V 0A6, Canada}
\email{javad.mashreghi@mat.ulaval.ca}

\address{Département de mathématiques et de statistique, Université Laval, Québec (QC), G1V 0A6, Canada}
\email{ryan.oloughlin.1@ulaval.ca}

\address{Département de mathématiques et de statistique, Université Laval, Québec (QC), G1V 0A6, Canada}
\email{thomas.ransford@mat.ulaval.ca}

\begin{abstract}
Given a compact convex planar domain $\Omega$ with non-empty interior, the classical Neumann's configuration constant $c_\RR(\Omega)$ is the norm of the Neumann-Poincaré operator $K_\Omega$ acting on the space of continuous real-valued functions on the boundary $\partial \Omega$, modulo constants. We investigate the related operator norm $c_\CC(\Omega)$ of $K_\Omega$ on the corresponding space of complex-valued functions, and the norm $a(\Omega)$ on the subspace of analytic functions. This change requires introduction of techniques much different from the ones used in the classical setting. We prove the equality $c_\RR(\Omega) = c_\CC(\Omega)$, the analytic Neumann-type inequality $a(\Omega) < 1$, and provide various estimates for these quantities expressed in terms of the geometry of $\Omega$. We apply our results to estimates for the holomorphic functional calculus of operators on Hilbert space of the type $\|p(T)\| \leq K \sup_{z \in \Omega} |p(z)|$, where $p$ is a polynomial and $\Omega$ is a domain containing the numerical range of the operator $T$. Among other results, we show that the well-known Crouzeix-Palencia bound $K \leq 1 + \sqrt{2}$ can be improved to $K \leq 1 + \sqrt{1 + a(\Omega)}$. In the case that $\Omega$ is an ellipse, this leads to an estimate of $K$ in terms of the eccentricity of the ellipse.
\end{abstract}

\thanks{This work has been done during Malman's visit at Département de mathématiques et de statistique, Université Laval, supported by Simons-CRM Scholar-in-Residence program. Mashreghi's research supported by an NSERC Discovery Grant and the Canada Research Chairs program. O'Loughlin supported by a CRM-Laval Postdoctoral Fellowship. Ransford's research supported by an NSERC Discovery Grant.}

\maketitle

\section{Introduction}

\label{S:IntroSec}

\subsection{Double-layer potential} 
Throughout this article, $\Omega$ will denote a compact convex planar domain with non-empty interior. If $C(\partial \Omega)$ is the space of continuous functions on the boundary $\partial \Omega$ and $f \in C(\partial \Omega)$, then its \textit{double-layer potential} $u$ is the harmonic function
\begin{equation}\label{E:DLformula}
u(z) = \frac{1}{\pi} \int_{\partial \Omega} f(\sigma) \, d \arg( \sigma - z) = \frac{1}{\pi} \int_{\partial \Omega} f(\sigma) \,\Re \Bigg( \frac{N(\sigma)}{\sigma - z} \Bigg) ds, \quad z \in \Omega^o.
\end{equation} Here $ds = |d\sigma|$ is the arclength measure on the rectifiable curve $\partial \Omega$, $\Omega^o$ is the interior of $\Omega$, and $N(\sigma)$ is the outer-pointing normal at the boundary point $\sigma$. The equality between the two expressions for $u(z)$ above follows from an elementary computation in the case that $\partial \Omega$ is sufficiently smooth. In the general case, we interpret $N(\sigma)(\sigma - z)^{-1}$ as a Borel measurable function on $\partial \Omega$. By convexity of the domain, both the tangent $T(\sigma)$ and the normal $N(\sigma)$ exist and are continuous at all but a countable number of points $\sigma$ which we will call \textit{corners}, at which the discontinuity of $T$ and $N$ amounts to a jump in the argument. In Appendix~\ref{sec:appendixA} we include more details regarding boundaries of planar convex domains, and other facts mentioned below.
 
The Neumann--Poincaré operator appears in connection with the study of boundary behaviour of the double-layer potential. It is known that $u$ given by \eqref{E:DLformula} has a continuous extension to $\partial \Omega$, and we have the representation
\begin{equation}\label{E:uKf-boundaryEq}
u(\zeta) = f(\zeta) + K_\Omega f(\zeta), \quad \zeta \in \partial \Omega
\end{equation}
where $K_\Omega$ denotes the \textit{Neumann-Poincaré integral operator}
\[ 
K_\Omega f(\zeta) := \frac{1}{\pi} \int_{\partial \Omega} f(\sigma)\, d\mu_\zeta( \sigma), \quad \zeta \in \partial \Omega.
\] Here $\mu_\zeta$ is the probability measure 
\begin{equation}\label{E:MuZetaEq}
d\mu_\zeta = (1-\theta_\zeta / \pi) d \delta_\zeta + \rho_\zeta ds
\end{equation} where $\theta_\zeta$ can be interpreted as the angle of the aperture at the possible corner at $\zeta$ of $\partial \Omega$, $\delta_\zeta$ is a unit mass at the point $\zeta$, and $\rho_\zeta$ is the Radon-Nikodym derivative 
\begin{equation}\label{E:MuZetaDensityFormula}
\rho_\zeta(\sigma) := \frac{d \mu_\zeta}{ds}(\sigma) := \frac{1}{\pi}\Re \Bigg( \frac{N(\sigma)}{\sigma - \zeta} \Bigg) = \frac{1}{\pi}\Im \Bigg( \frac{T(\sigma)}{\sigma - \zeta}\Bigg).
\end{equation} It is natural to use the convention that $\theta_\zeta = \pi$ if $\zeta$ is not a corner. This occurs precisely when $\mu_\zeta$ assigns no mass to the singleton $\{\zeta\}$. We will say that the collection of measures $\{\mu_\zeta\}_{\zeta\in \partial \Omega}$ is the \textit{Neumann-Poincaré kernel} of $\Omega$.

The density $\rho_\zeta$ has the following useful geometric interpretation. If $\sigma \in \partial \Omega \setminus \{\zeta\}$ is not a corner, and $R_{\zeta, \sigma}$ is the radius of the unique circle passing through $\zeta$ which is tangent to $\partial \Omega$ at $\sigma$, then the equality
\begin{equation}
\label{E:DensityCircleRadiusFormula}
\rho_\zeta(\sigma) = \frac{1}{2 \pi R_{\zeta,\sigma}}
\end{equation} holds. The radius $R_{\zeta, \sigma}$ may degenerate to $\infty$ if $\zeta$ is contained in the tangent line to $\partial \Omega$ passing through $\sigma$. In that case we see easily that $\rho_\zeta(\sigma) = 0$, so \eqref{E:DensityCircleRadiusFormula} still holds. To establish the formula, note that the center $m$ of the circle in question is of the form $m = \sigma - R N(\sigma)$, where the radius $R = R_{\zeta,\sigma} > 0$ of the circle satisfies $|m - \zeta|^2 = |(\sigma - \zeta) - R N(\sigma)|^2 = R^2$. Expanding the squares and solving for $R$ leads to \eqref{E:DensityCircleRadiusFormula}. 

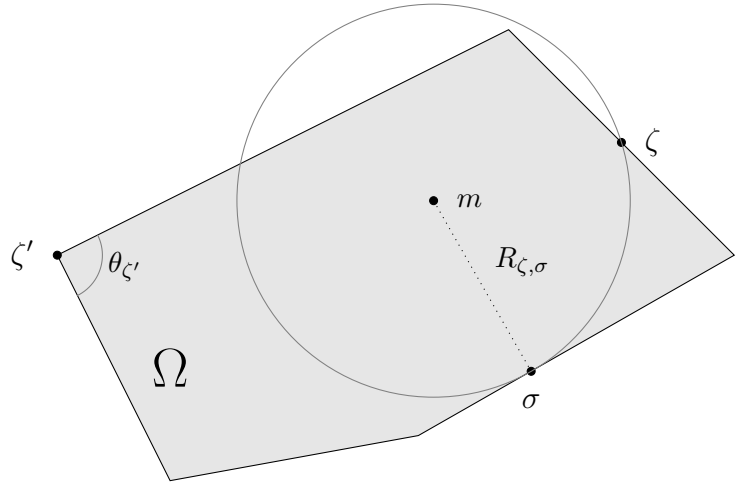
\begin{figure}
\centering
\begin{tikzpicture}[scale=3]
    \draw[fill=gray!20] (1,0) -- (0, 01) -- (-2, 0) -- (-1.5, -1) -- (-0.4, -0.8) --(1, 0) -- cycle;
    \node at (-1.5, -0.5) {\huge $\Omega$};
    \node at (-1.7, -0.05) {$\theta_{\zeta'}$};
    \filldraw (-2, 0) circle (0.5 pt) node[left, xshift=-5pt]{\large $\zeta'$};
    \filldraw (0.5, 0.5) circle (0.5 pt) node[right, xshift= 5pt]{\large $\zeta$};
    \filldraw (0.1, -0.515) circle (0.5 pt) node[below, yshift=-5pt]{\large $\sigma$};
    \filldraw (-0.3324, 0.2417) circle (0.5 pt) node[right, xshift=5pt]{$m$};
    %\filldraw (-0.285, -0.5) circle (0.5 pt) node[left]{$\sigma$};
    \draw [gray, domain=-64:27] plot ({0.2*cos(\x)-2}, {0.2*sin(\x)});
    \draw[color=gray] (-0.3324, 0.2417) circle (0.8715);
    \draw[ dotted] (-0.3324, 0.2417) -- (0.1, -0.515) node[midway, above, xshift= 15 pt]{$R_{\zeta, \sigma}$};
\end{tikzpicture}
\caption{Example domain $\Omega$ with corner of angle $\theta_{\zeta'}$ at $\zeta'$, and a circle of radius $R_{\zeta, \sigma}$ with center $m$, tangent to $\partial \Omega$ at $\sigma$ and passing through $\zeta$.}
\label{fig:fig0}
\end{figure}

\subsection{Neumann's configuration constant}

\subsubsection{Real configuration constant} Historically, the Neumann--Poincaré operator has been used to solve the \textit{Dirichlet problem} of finding a harmonic extension to $\Omega^o$ of a given continuous function $u$ on $\partial \Omega$. The extension can be obtained by finding $f \in C(\partial \Omega)$ which solves \eqref{E:uKf-boundaryEq}. Indeed, if such an $f$ is found, then the extension of $u$ to $\Omega^o$ is given by the double-layer potential in \eqref{E:DLformula}. This naturally leads to questions of invertibility of the operator $I + K_\Omega$ appearing on the right-hand side of \eqref{E:uKf-boundaryEq}, and consequently to the introduction of the \textit{Neumann's configuration constant}, which we shall soon define as the operator norm of $K_\Omega$ acting on an appropriate space. Note that if $\mathbf{1}$ is the constant function, then we have that $K_\Omega \mathbf{1} = \mathbf{1}$, since each $\mu_\zeta$ is a probability measure. Thus $K_\Omega$ can be naturally defined as a linear mapping on the quotient space $C(\partial \Omega)/\CC \mathbf{1}$. The classical approach is to instead consider $K_\Omega$ as acting on the space of \textit{real-valued} continuous functions $C_\RR(\partial \Omega)$, in which case the corresponding quotient space $C_\RR(\partial \Omega)/\RR \mathbf{1}$ is endowed with the norm 

\begin{equation} \label{E:RealQuotientNorm}
\| g + \RR \mathbf{1} \|_{\partial \Omega} := \max_{\zeta, \zeta' \in \partial \Omega} \frac{|g(\zeta) - g(\zeta')|}{2} = \min_{r \in \RR} \, \max_{\zeta \in \partial \Omega} |g(\zeta) - r|.
\end{equation}
It is not hard to see that the two above expressions for the norm of the coset $g + \RR \mathbf{1}$ are equivalent: they are both equal to half of the length of the interval $g(\partial \Omega) := \{ g(\partial \Omega) : \zeta \in \partial \Omega\}$, the image of $g$. The right-most expression is minimized by choosing $r$ to be the mid-point of the image interval. Neumann's (real) configuration constant $c_\RR(\Omega)$ is defined as the operator norm of $K_\Omega$ acting on the quotient space $C_\RR(\partial \Omega)/ \RR \mathbf{1}$:
\begin{equation}\label{E:RconfigConstDef}
c_\RR(\Omega) :=  \| K_\Omega : C_\RR(\partial \Omega)/ \RR \mathbf{1} \to C_\RR(\partial \Omega)/ \RR \mathbf{1} \|.
\end{equation} 
It is not hard to see that we may let $K_\Omega$ instead act from $C_\RR(\partial \Omega)$ into the quotient $C_\RR(\partial \Omega) / \RR \mathbf{1}$ without affecting the operator norm. Since each measure $\mu_\zeta$ is of unit mass, we have $0 \leq c_\RR(\Omega) \leq 1$. If 
\[ \|f \|_{\partial \Omega} := \sup_{\zeta \in \partial \Omega} |f(\zeta)| \leq 1, \]
then \[|K_\Omega f(\zeta) - K_\Omega f(\zeta')| \leq \| \mu_{\zeta} - \mu_{\zeta'}\|,\] where we use the total variation norm (functional norm) on the right-hand side. By varying $f$ over the unit ball of $C_\RR(\partial \Omega)$ and $\zeta, \zeta'$ over $\partial \Omega$, we obtain the important relation
\begin{equation} \label{E:RconfigConstantFormula}
c_\RR(\Omega) := \sup_{\zeta, \zeta' \in \partial \Omega} \frac{\| \mu_{\zeta} - \mu_{\zeta'}\|}{2}.
\end{equation}
This expression for $c_\RR(\Omega)$ will play a fundamental role in our study.

\subsubsection{Neumann's lemma}
From \eqref{E:RconfigConstantFormula} we can immediately deduce that $c_\RR(\Omega) = 1$ in the case that $\Omega$ is a triangle or a convex quadrilateral. Indeed, in those cases one sees from \eqref{E:MuZetaEq} and \eqref{E:MuZetaDensityFormula} that if $\zeta_1$ and $\zeta_2$ are corners of $\Omega$ (opposing, in the case of the quadrilateral) then $\mu_{\zeta_1}$ and $\mu_{\zeta_2}$ are mutually singular, and so $\| \mu_{\zeta_1} - \mu_{\zeta_2}\| = 2$, implying $c_\RR(\Omega) = 1$. \textit{Neumann's lemma}, which appears initially in Neumann's book \cite{Ne1877}, states that the cases of the triangle and quadrilateral are exceptional. For any other type of domain we have the strict inequality $c_\RR(\Omega) < 1$. See \cite{Sch68} for a proof of this claim by Schober, and the curious history of incomplete attempts at a valid proof in full generality. Neumann's lemma implies the invertibility of $I + K_\Omega$ on $C_\RR(\partial \Omega)/ \RR \mathbf{1}$, and thus the solvability of the Dirichlet problem on a convex domain $\Omega$ which is not one of the two exceptional cases. The remaining cases can be handled by considering instead powers of $K_\Omega$. See, for instance, \cite[Theorem 3.8]{Kr80}, \cite[Proposition 7]{DD99}, or the article \cite{PS05}, which contains also an exposition of the double-layer potential and Neumann's lemma.

At the other extreme, we have $c_\RR(\Omega) = 0$ if and only if $\Omega$ is a disk. This result will be proved in Section~\ref{S:ExamplesSection}.

\subsection{Complex and analytic configuration constants}

\subsubsection{Two new configuration constants} In the present article we will discuss certain applications of the double-layer potential to operator theory which motivate the definition of the \textit{complex configuration constant}
\begin{equation}\label{E:CconfigConstDef}
c_\CC(\Omega) :=  \| K_\Omega : C(\partial \Omega)/ \CC \mathbf{1} \to C(\partial \Omega)/ \CC \mathbf{1} \|.
\end{equation} The difference between \eqref{E:RconfigConstDef} and \eqref{E:CconfigConstDef} is that the latter is the norm of $K_\Omega$ on the larger space of complex-valued functions. As a consequence, we have $c_\RR(\Omega) \leq c_\CC(\Omega) \leq 1$. There is a principal difference between the geometric interpretations of the norms in the quotient spaces $C_\RR(\partial \Omega) / \RR \mathbf{1}$ and $C(\partial \Omega) / \CC \mathbf{1}$. In the former case, as we have already noted, the norm \eqref{E:RealQuotientNorm} of the coset represented by the real-valued function $g$ is equal to half of the length of the image of $g$, this image being an interval on the real line $\RR$. In the case of complex-valued $g$, the quotient norm 
\begin{equation}\label{E:COmplexQuotientNorm}
\| g + \CC \mathbf{1}\|_{\partial \Omega} := \min_{\lambda \in \CC} \max_{\zeta \in \partial \Omega} | g(\zeta) - \lambda|
\end{equation} can instead be interpreted as the \textit{radius of the smallest disk containing the image of $g$}. A crucial difference is that we lose the ability to estimate the norm of the coset $g + \CC \mathbf{1}$ by considering the quantities $|g(\zeta) - g(\zeta')|$ only. This is the essence of why new tools are required to treat this case.

We will also study an analogous analytic constant, which is the norm of the operator $K_\Omega$ restricted to the subspace of analytic functions in $C(\partial \Omega)$. More precisely, we let $\A(\Omega)$ be the space of functions which are continuous in $\Omega$ and analytic in $\Omega^o$. Each function in $\A(\Omega)$ has a unique restriction to $\partial \Omega$, and thus $\A(\Omega)$ can be naturally identified with a subspace of $C(\partial \Omega)$. We define the \textit{analytic configuration constant} as
\begin{equation}\label{E:AconfigConstDef}
a(\Omega) :=  \| K_\Omega : \A(\Omega) / \CC\mathbf{1} \to C(\partial \Omega) / \CC\mathbf{1} \|.
\end{equation}
The space $\A(\Omega)$ is not invariant under $K_\Omega$, but we do have that $K_\Omega f$ is the complex conjugate of a function in $\A(\Omega)$ (in \cite[proof of Lemma 2.1]{CP17} this claim is established for $\Omega$ with smooth boundary, but the same argument works in general). Clearly, we have the inequality $a(\Omega) \leq c_\CC(\Omega)$. We note also that if $\widetilde{\Omega}$ is the image of $\Omega$ under an affine transformation of the plane, then the configuration constants of the two domains are equal. We shall verify this claim in Section~\ref{S:ApplNumRan}.

\subsubsection{An application to functional calculi} Given an operator $T$ on a Hilbert space $\mathcal{H}$ with \textit{numerical range}
\[
W(T) := \{ \langle Tx, x \rangle_\HH : x \in \HH, \|x\|_{\HH} = 1\},
\] we are interested in the optimal constant $K > 0$ in the inequality
\begin{equation}\label{E:PolFuncCalcBound}
\|p(T)\| \leq K\cdot \sup_{z \in W(T)} |p(z)| = K \|p\|_{W(T)},
\end{equation}
where $p$ is an analytic polynomial, and the left-hand side is the operator norm of $p(T)$ acting on $\mathcal{H}$. More generally, if $W(T)$ in \eqref{E:PolFuncCalcBound} is replaced by an arbitrary domain $\Omega$, and if the corresponding inequality holds for some $K$, then we say that $\Omega$ is a \textit{$K$-spectral set} for $T$. Von Neumann's inequality says that the unit disk is a $1$-spectral set for any contraction $T$, and a result of Okubo-Ando from \cite{OA75} says that any disk containing $W(T)$ is a $2$-spectral set for $T$.

The numerical range $W(T)$ is a bounded convex subset of the plane, its closure $\overline{W(T)}$ contains the spectrum $\sigma(T)$ of $T$, and it has non-empty interior in the case that $T$ is not a normal operator (see, for instance, \cite[Chapter 1]{GR96}). For normal operators, the bound \eqref{E:PolFuncCalcBound} with constant $K = 1$ is a consequence of the spectral theorem, and it suffices to take the supremum on the right-hand side over the smaller set $\sigma(T)$. For general $T$, even establishing the existence of a bound as in \eqref{E:PolFuncCalcBound} is a non-trivial task. A result of Delyon-Delyon from \cite[Theorem 3]{DD99} establishes the existence of the bound, and shows that $K$ can be chosen depending only on the area and the diameter of $W(T)$. The remarkable work of Crouzeix in \cite{Cr07} establishes that \eqref{E:PolFuncCalcBound} holds with $K \leq 11.08$. A subsequent work of Crouzeix and Palencia in \cite{CP17} improves the estimate to $K \leq 1 + \sqrt{2}$. The Neumann--Poincaré operator appears as an essential tool in all of the mentioned works.
The standing conjecture of Crouzeix from \cite{Cr04} is that the bound holds with $K = 2$. This bound is presently known to hold in the case $\mathcal{H}$ being of dimension 2, and has been established by Crouzeix in \cite{Cr04}.

Our interest in the new notions of configuration constants is inspired by a recent work of Schwenninger and de Vries in \cite{SdV23}, where bounds for general homomorphisms between uniform algebras and the algebras of bounded linear operators are studied. In Section~\ref{S:ApplNumRan} we will combine their arguments with the methods of Crouzeix-Palencia to obtain the following estimate:
\begin{equation}\label{E:PolFuncCalcBoundAOmega}
\|p(T)\| \leq \bigg( 1 + \sqrt{1 + a(W)} \bigg)  \| p \|_W, \quad W := \overline{W(T)}.
\end{equation} For instance, if $W$ is a disk, then $a(W) = 0$, which gives the Okubo-Ando result mentioned above. In \cite{SdV23}, Schwenninger and de Vries recovered this result also. The estimate \eqref{E:PolFuncCalcBoundAOmega} is our motivation for the following investigation of the configuration constants $c_\RR(\Omega), c_\CC(\Omega)$ and $a(\Omega)$, and the relations between them.

\subsection{Main results}

\subsubsection{Relation between the real and complex constants} 

Consider the situation in Figure~\ref{fig:fig1}, where the triangular image of the complex-valued function $g: \partial \Omega \to \CC$ is contained in a disk of radius $1$, and intersects the boundary circle of the disk in three distinct points. The three-point set $\{g(\zeta_1), g(\zeta_2), g(\zeta_3)\}$ is not contained in any open half-circle of the boundary, and it follows from a simple geometric argument (which we shall present in the proofs below) that $\| g + \CC \mathbf{1}\|_{\partial \Omega} = 1$. However, the sides of the triangular image of $g$ are all of lengths strictly less than $2$, and this implies that \[\| g + \CC \mathbf{1}\|_{\partial \Omega} = 1 > \max_{\zeta, \zeta' \in \partial \Omega} \frac{|g(\zeta)- g(\zeta')|}{2}.\] If such a function $g$ lies in the image of the unit ball of $C(\partial \Omega)$ under the Neumann--Poincaré operator $K_\Omega$ for some domain $\Omega$ which satisfies $c_\RR(\Omega) < 1$, then a strict inequality $c_\RR(\Omega) < c_\CC(\Omega)$ occurs. Our first main result excludes this possibility, and so establishes the simplest possible relation between the real and complex configuration constants.

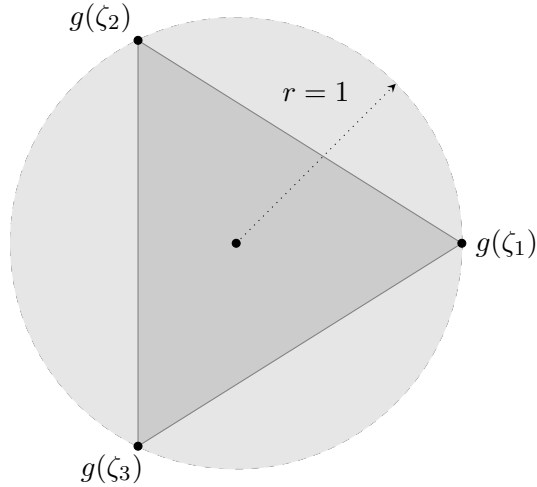
\begin{figure}
    \centering
    
    \begin{tikzpicture}[scale=3]
    \draw[dashed] (0,0) circle (1);
    \filldraw[gray!20] (0, 0) circle (1);
    
	\filldraw[draw=black!50, fill=gray!40] (-0.435, 0.9) -- (1,0) -- (-0.435, -0.9) -- cycle;
   
    \filldraw (0,0) circle (0.5pt);
    \draw[-stealth, dotted] (0,0) -- (45:1) node[midway,above, yshift=20]{$r = 1$};
    
    \node at (1.20, 0) {$g(\zeta_1)$};
    \filldraw (1, 0) circle (0.5pt);  
    
    \node at (-0.60, 1) {$g(\zeta_2)$};
    \filldraw (-0.435, 0.9) circle (0.5pt);  
    
    \node at (-0.55, -1) {$g(\zeta_3)$};
    \filldraw (-0.435, -0.9) circle (0.5pt);       

\end{tikzpicture}

     \caption{A triangular image of a complex-valued function $g$ contained in a disk of radius $1$, with three points on the boundary of a disk.}
    \label{fig:fig1}
\end{figure}

\begin{theorem} \label{T:MainTheorem1}
The equality 
\[
c_\RR(\Omega) = c_\CC(\Omega)
\] holds for every compact convex domain $\Omega$ with non-empty interior.
\end{theorem}

It follows that every considered domain has a well-defined \textit{configuration constant} $c(\Omega)$, which is equal to the operator norm of $K_\Omega$ on $C(\partial \Omega) / \CC \mathbf{1}$, and which can be computed according to the right-hand side of \eqref{E:RconfigConstantFormula}. An important consequence of this result is the inequality 
\begin{equation} \label{E:AconstEstimateByRconst}
a(\Omega) \leq c(\Omega) = \sup_{\zeta, \zeta' \in \partial \Omega} \frac{\| \mu_{\zeta} - \mu_{\zeta'}\|}{2},
\end{equation} which, as we shall soon see, has some interesting consequences.

Theorem~\ref{T:MainTheorem1} doesn't appear nearly as straightforward to prove as it is to state, and the proof takes up a large portion of the article. However, the only property of the Neumann-Poincaré operator used in the proof is that its integral kernel $\{\mu_\zeta\}_{\zeta \in \partial \Omega}$ consists of real-valued measures. In fact, the theorem will be deduced as a corollary of a result which we call the \textit{Three-measures theorem}, and which is a general statement regarding the geometry of the space $C(X)$ of continuous functions on a compact Hausdorff space $X$. This result, which we discuss and prove in Section~\ref{S:3MSection}, puts a restriction on the possible configurations of point sets in the plane which arise as values of a collection of real-valued functionals on $C(X)$. 

\subsubsection{Analytic Neumann's lemma} Note that the above estimate in \eqref{E:AconstEstimateByRconst}, together with Neumann's lemma, implies that $a(\Omega) < 1$ whenever $\Omega$ is not a triangle or a quadrilateral. This can be improved, for we have an analytic version of Neumann's lemma, in which no exceptional cases occur. 

\begin{theorem} \label{T:Maintheorem2} The strict inequality
\[ a(\Omega) < 1 \] holds for every compact convex domain $\Omega$ with non-empty interior.
\end{theorem}

Our proof of Theorem~\ref{T:Maintheorem2} is much different from the one given by Schober in his proof of the \textit{real} Neumann's lemma in \cite{Sch68}, but it works also in the real context. At the end of Section \ref{S:ProofT2} we show how our technique leads to a different proof of Neumann's lemma. 

\subsubsection{Functional calculus bounds}

The following result has already been mentioned above.

\begin{theorem} \label{T:Maintheorem3}
Let $T: \HH \to \HH$ be a bounded linear operator on a Hilbert space $\HH$ with numerical range $W(T)$ which has non-empty interior. Then, for every polynomial $p$, we have
\[ \| p(T)\| \leq \Bigl( 1 + \sqrt{1 + a(W)} \Bigr) \| p \|_{W(T)}. \]
\end{theorem}

Recall that if the numerical range of an operator has empty interior, then the operator is normal, and so \eqref{E:PolFuncCalcBound} holds with $K = 1$. From this observation and Theorem~\ref{T:Maintheorem2} we obtain that for any fixed operator $T:\HH \to \HH$, the optimal constant $K$ in \eqref{E:PolFuncCalcBound} is always strictly smaller than $1 + \sqrt{2}$. In fact, we deduce from our results that we have the inequality
\[
\|p(T)\| \leq K_{W}  \| p \|_W
\] with a constant 
\[K_W < 1 + \sqrt{2}\]
which depends only on the shape of $W = \overline{W(T)}$, and not on the operator $T$ itself. We show in Section~\ref{S:ExamplesSection} that no better universal bound can be obtained by means of the analytic configuration constant: for any $\epsilon > 0$ there exists a “thin" quadrilateral $\Omega_\epsilon$ for which we have $a(\Omega_\epsilon) > 1 - \epsilon$. However, fixing the dimension of the Hilbert space $\HH$, one may combine earlier results of Crouzeix to obtain a uniform improvement. The optimal constant $K$ in \eqref{E:PolFuncCalcBound} varies with $T$, and we may consider the supremum of these quantities among all operators $T$ on a Hilbert space $\HH$ of a fixed dimension $N$. In \cite[Theorem 2.2]{Cr16}, Crouzeix proved that there exists an operator realizing this supremum. An immediate corollary of his result, Theorem~\ref{T:Maintheorem2} and Theorem~\ref{T:Maintheorem3} is the following.

\begin{corollary}
For every positive integer $N$, there exists a constant $C_N < 1 + \sqrt{2}$ for which we have 
\[ \| p(T)\| \leq C_N \|p\|_{W(T)}\] whenever $T$ is an operator on an $N$-dimensional Hilbert space, and $p$ is a polynomial.
\end{corollary}

This improves the Crouzeix-Palencia bound, although by an indefinite amount. 

\subsubsection{Estimates for the configuration constants}
In Section~\ref{S:ExamplesSection} we present also other computations and estimates for the configuration constants. Surprisingly, in the case of an elliptical domain, the configuration constant is computable exactly, and we obtain
\[ c(\Omega_{a,b}) = \frac{2}{\pi}\arctan\Bigl(\frac{1}{2}\Bigl|\frac{b}{a}-\frac{a}{b}\Bigr|\Bigr) =  \frac{2}{\pi}\arctan\Bigl( \frac{1}{2} \frac{e^2}{\sqrt{1-e^2}}\Bigr)\] where $a$ and $b$ are lengths of the semi-axes of the ellipse $\Omega_{a,b}$, and $e$ is the eccentricity of the ellipse, given by $e := \sqrt{1 - b^2/a^2}$ in case that $a \geq b$. This fact, together with Theorem~\ref{T:MainTheorem1}, estimate \eqref{E:PolFuncCalcBoundAOmega}, and the inequality $a(\Omega) \leq c(\Omega)$, has the following consequence.

\begin{corollary}\label{C:EllipseSpecConstEstimate} Let $T: \HH \to \HH$ be a bounded linear operator on a Hilbert space $\HH$ with numerical range contained in (or equal to) the ellipse $\Omega_{a,b}$. Then, for every polynomial $p$, we have 
\[
\|p(T)\| \leq K(a,b) \| p \|_{\Omega_{a,b}}.
\] where
\[ K(a,b) := 1 + \sqrt{1 + \frac{2}{\pi}\arctan\Bigl(\frac{1}{2}\Bigl|\frac{b}{a}-\frac{a}{b}\Bigr|\Bigr)}.\]
\end{corollary}
Note that the function $a \mapsto K(a,1)$ is continuous and increasing for $a \geq 1$, and we have
\[ \lim_{a \to \infty} K(a,1) = 1 + \sqrt{2}, \quad \lim_{a \to 1} K(a,1) = 2.\]
Hence the estimate in Corollary~\ref{C:EllipseSpecConstEstimate} gets worse as the eccentricity of the ellipse $\Omega_{a,b}$ grows, and approaches the Crouzeix-Palencia bound in the limit $a \to \infty$. On the other hand, as $a \to 1$, the eccentricity of the ellipse $\Omega_{a,1}$ tends to $0$. The estimate is then close to the conjectured optimal bound $K = 2$ and coincides with the Okubo-Ando bound for $a=1$, in which case the domain is a disk. From this perspective, Corollary~\ref{C:EllipseSpecConstEstimate} may be interpreted as an elliptical generalization of the Okubo-Ando estimate.

For many other types of domains, the exact value of $c(\Omega)$ is inaccessible. To help the situation, we establish an integral estimate which gives an upper bound on $c(\Omega)$ in terms of the curvature of $\partial \Omega$, roughly speaking. For a fixed $\sigma$ which is not a corner of $\partial \Omega$, recall the definition of $R_{\zeta, \sigma}$ in \eqref{E:DensityCircleRadiusFormula}, and consider
\begin{equation}
\label{E:RSigmaDefIntro}
R_\Omega(\sigma) := \sup_{\zeta \in \partial \Omega} R_{\zeta, \sigma}.
\end{equation}
If $\kappa(\sigma)$ is the curvature of $\partial \Omega$ at $\sigma$, then $R_\Omega(\sigma)$ is at least as large as the \textit{radius of curvature} \begin{equation}
\label{E:RCurvDefIntro}1/\kappa(\sigma) = \lim_{\zeta \to \sigma} R_{\zeta, \sigma},
\end{equation} which is also the radius of the osculating circle at $\sigma$. Geometrically, $R_\Omega(\sigma)$ is the radius of the smallest disk tangent to $\partial \Omega$ at $\sigma$ which contains $\Omega$, if such a disk exists, and it is equal to $\infty$ otherwise. The latter case occurs, for instance, if $\sigma$ lies on a straight line segment contained in $\partial \Omega$. However, if $\partial \Omega$ is sufficiently curved on a segment of $\partial \Omega$, then $R_\Omega$ will be bounded above there. We obtain in such a situation a non-trivial upper bound on $c(\Omega)$.

\begin{theorem} \label{T:ROmegaTheorem}
With the above notation, we have the estimate
\[ c(\Omega) \leq 1-\frac{1}{2\pi}\int_{\partial\Omega}\frac{ds}{R_\Omega}.\]
\end{theorem}

The result implies spectral constant estimates similar to the one in Corollary~\ref{C:EllipseSpecConstEstimate} above. It also generalizes some similar results in the literature. See Section~\ref{S:ExamplesSection} for further details and examples.

\subsubsection{An unresolved matter}

We have mentioned above that $c(\Omega) = 0$ if and only if $\Omega$ is a disk. With some additional effort, we will show in Section~\ref{S:ExamplesSection} that the condition $a(\Omega) = 0$ also characterizes disks. In this case, we have the equality $a(\Omega) = c(\Omega)$. It is natural to ask whether other domains exist for which the equality occurs, or if the case of the disk is exceptional.

\begin{question*} Do we always have the strict inequality
\[ a(\Omega) < c(\Omega)\] whenever $\Omega$ is not a disk?
\end{question*}

As a consequence of Theorem~\ref{T:Maintheorem2} and the exceptional cases of Neumann's lemma, we see that the strict inequality holds whenever $\Omega$ is a triangle or a quadrilateral. The authors have not been able to confirm that the inequality holds in any other examples.

\subsection{Notations} Some of our notation has already been introduced above. For a continuous function $f$ defined on a set $X$, we denote by $\| f \|_X$ the supremum of $|f|$ over $X$. For cosets of the form $f + \CC\mathbf{1}$ we use the convention \[\| f + \CC \mathbf{1}\|_X := \inf_{\lambda \in \CC} \|f + \lambda \|_X,\] with similar convention for real-valued $f$ and cosets $f + \RR\mathbf{1}$. A norm $\| \cdot \|$ without a subscript usually denotes a linear functional norm or a total variation norm of a measure. The distinction will be unimportant and should anyway be easy to deduce from context. We use boldface letters, such as $\textbf{x}$, to denote vectors in $\RR^n$, and plain letters, such as $x_j$, to denote the coordinates.

\section{The three-measures theorem}
\label{S:3MSection}

\subsection{Definitions of relevant spaces and operators} \label{S:Section2Definitions}

Theorem~\ref{T:MainTheorem1} will be proved as a corollary of our analysis of three-point configurations 
\[ \big( \ell_1(x), \ell_2(x), \ell_3(x) \big) \in \CC^3, \] where $x$ is an element of a given normed space $\NN$, and $\ell_1, \ell_2, \ell_3 \in \NN^*$ are three bounded linear functionals on $\NN$. A point configuration of this type has to satisfy certain conditions. For instance, we must have the distance bound 
\[
|\ell_j(x) - \ell_k(x)| \leq \|\ell_j - \ell_k\|_{\NN^*} \|x\|_{\NN}, \quad 1 \leq j,k \leq 3.
\]
Our principal interest will be in estimating the radius of the smallest disk which contains such a three-point set. 

In order to use the tools of functional analysis, we will formulate our problem as one of estimating the norm of an operator between normed spaces. To this end, we use the space $\CC^3$ of triples of complex numbers, and we equip it with the following norm:

\begin{equation}
\label{E:C3supnorm}
\| (a,b,c) \|_\infty := \max \{|a|, |b|, |c| \}.
\end{equation}
Similarly to our previous notational conventions, we shall set $\mathbf{1} := (1,1,1) \in \CC^3$. The quotient norm in the quotient space $\CC^3 / \CC \mathbf{1}$ satisfies 
\[ \| (a,b,c) + \CC \mathbf{1} \|_{\infty} := \min_{\lambda \in \CC} \, \max \{ |a-\lambda|, |b-\lambda|, |c-\lambda|\},
\] and it has the geometric interpretation adequate to our problem: it is the radius of the smallest disk containing the three point set $\{a,b,c\}$. Given a normed space $\NN$ and three linear functionals $\ell_1, \ell_2, \ell_3 \in \NN^*$, we introduce the linear operator $\LL : \NN \to \CC^3 / \CC \mathbf{1}$ defined by
\begin{equation}
\label{E:LOperatorDef} \LL x := \big( \ell_1(x), \ell_2(x), \ell_3(x) \big) + \CC \mathbf{1}.
\end{equation}
With these conventions, each three-point configuration $\big( \ell_1(x), \ell_2(x), \ell_3(x) \big)$ is contained in a disk of radius at most $\|\LL\|_{\NN \to \CC^3 / \CC \mathbf{1}} \cdot \|x\|_\NN$. We want to estimate the operator norm $\|\LL\|_{\NN \to \CC^3 / \CC \mathbf{1}}$. 

\subsection{Statement of the theorem}

Without any information regarding the space $\NN$ or the functionals $\ell_1, \ell_2, \ell_3$, the optimal estimate is 
\begin{equation}
\label{E:3MComplexBestEst}
\|\LL\|_{\NN \to \CC^3 / \CC \mathbf{1}} \leq \frac{1}{\sqrt{3}} \max_{j,k} \|\ell_j - \ell_k\|.
\end{equation}
Indeed, we see that we cannot do better by choosing $\NN = \CC$, $x = 1$, and the functionals (scalars) to be the vertices of an equilateral triangle inscribed in the unit circle. For instance, 
\[ \ell_1 = 1, \quad \ell_2 = -1/2 + i\sqrt{3}/2,  \quad \ell_3 = -1/2 - i\sqrt{3}/2 .\] The sides of the triangle have the common length equal to $| \ell_i - \ell_j| = \sqrt{3}$, and the smallest disk containing the three points $\ell_i(x) = \ell_i$ is the unit disk itself. Thus, in this case, \eqref{E:3MComplexBestEst} holds with equality. The estimate holds in general as a consequence of \textit{Jung's theorem}, which appeared first in \cite{Ju1910}, and which in the context of the plane says that any set of diameter $d$ is contained in a disk of radius $d/\sqrt{3}$. In our setting $d \leq \max_{j,k} \|\ell_j - \ell_k\|_{\NN^*}$, and so the estimate \eqref{E:3MComplexBestEst} follows from Jung's theorem.

In our intended application, the role of the space $\NN$ is played by $C(X)$, the Banach space of continuous functions on a compact Hausdorff space $X$, and the functionals are given by integration against \textit{real-valued} measures
\[ f \mapsto \mu_j(f) := \int_X f \, d \mu_j.\] It turns out that the three-point configurations which arise in this way are contained in disks of radius smaller than predicted by Jung's theorem. The main result of the section is the following.

\begin{theorem} \label{T:3MT}
Let $C(X)$ be the space of continuous functions on a compact Hausdorff space $X$, and $\LL: C(X) \to \CC^3 / \CC \mathbf{1}$ be the operator in \eqref{E:LOperatorDef} defined by three functionals induced by three finite real-valued Borel measures $\mu_1, \mu_2, \mu_3$. Then 
\begin{equation}
\label{E:3MTEstimate}
\|\LL\|_{C(X) \to \CC^3 / \CC \mathbf{1}} = \frac{1}{2} \max_{j,k} \|\mu_j - \mu_k\|.
\end{equation}
\end{theorem}

It is the "$\leq$" estimate in \eqref{E:3MTEstimate} that is the critical one. The lower bound "$\geq$" follows from the definition of the functional norm
\[ \frac{1}{2} \| \mu_j - \mu_k\| = \frac{1}{2}\sup_{f : \|f\|_X = 1} | \mu_j(f) - \mu_k(f)| \leq \|\LL\|_{C(X) \to \CC^3 / \CC \mathbf{1}}. \]

We will spend the rest of the section on proving Theorem~\ref{T:3MT}. The outline of the proof is as follows. We will first use duality to formulate the problem in terms of the adjoint operator $\LL^*$ between the dual spaces. Next, a discretization will help us reduce the dual problem to a finite-dimensional optimization problem. Finally, we will solve the finite-dimensional problem by the use of techniques of convex analysis. 

Before proceeding, we remark that the natural generalization of the above theorem to an arbitrary $n$-tuple of real-valued measures is valid. See Theorem~\ref{T:NMT} below.

\subsection{Dual problem}

Let us denote by $Y$ the space $\CC^3/\CC \mathbf{1}$ equipped with the norm in \eqref{E:C3supnorm}. Then the dual space $Y^*$ is the two-dimensional space of three-tuples $(\alpha,\beta, \gamma)$ of complex numbers which satisfy
\[ \alpha+\beta+\gamma = 0,\] and the norm on $Y^*$ is given by
\[ \|(\alpha,\beta,\gamma)\|_1 := |\alpha| + |\beta| + |\gamma|.\] In the case $\NN = C(X)$, the dual space $(C(X))^*$ is just the space of finite Borel measures on $X$. The adjoint operator $\LL^* : Y^* \to \NN^*$ is then given by 
\[ \LL^* : (\alpha,\beta,\gamma) \mapsto \alpha\mu_1 + \beta\mu_2 + \gamma\mu_3\] and the estimate \eqref{E:3MTEstimate} is equivalent to 
\begin{equation}
\| \alpha \mu_1 + \beta \mu_2 + \gamma \mu_3\| \leq \frac{|\alpha| + |\beta| + |\gamma|}{2} \max_{j,k} \|\mu_j - \mu_k\|.
\end{equation} Since $\alpha + \beta = -\gamma$ and $(\alpha + \beta + \gamma)\mu_3 = 0$, we may rewrite the above inequality into
\begin{equation} 
\| \alpha \nu_1 + \beta \nu_2\| \leq \frac{|\alpha| + |\beta| + |\alpha + \beta|}{2} \max \Bigl\{ \|\nu_1\|, \| \nu_2\|, \|\nu_1 - \nu_2\| \Bigr\}
\end{equation} where 
\[
\nu_1 := \mu_1 - \mu_3, \quad \nu_2 := \mu_2 - \mu_3. \] Note that $\nu_1$ and $\nu_2$ are real-valued if $\mu_1, \mu_2, \mu_3$ are real-valued. Theorem~\ref{T:3MT} is thus a consequence of the following slightly more general statement in which the topological structure of $X$ does not play a role.

\begin{proposition} \label{P:3MTDualProblem}
Let $\nu_1$ and $\nu_2$ be two finite real-valued measures on a measurable space $X$. Then for any complex numbers $\alpha, \beta$ we have the inequality
\begin{equation}\label{E:3m}
\|\alpha\nu_1+\beta\nu_2\|\le \frac{|\alpha|+|\beta|+|\alpha + \beta|}{2}\max \Bigl\{ \|\nu_1\|, \| \nu_2\|, \|\nu_1 - \nu_2\| \Bigr\},
\end{equation} where the norm on the right-hand side is the total variation norm $\|\mu\| := |\mu|(X)$. 
\end{proposition}

In our next step, we shall simplify the problem further, and show that Proposition~\ref{P:3MTDualProblem} can be established by considering finite sets $X$ only.

\subsection{Discretization} 
With notations as in Proposition~\ref{P:3MTDualProblem}, set $\sigma:=|\nu_1|+|\nu_2|$. Then $\sigma$ is a positive finite measure on $X$, and by the Radon--Nikodym theorem
we have $d\nu_1=f\,d\sigma$ and $d\nu_2=g\,d\sigma$, where $f,g$ are bounded real measurable functions
on~$X$. For a moment, let $\|\cdot \|_{\sigma,1}$ denote the norm 
\[ \|f \|_{\sigma,1} := \int_X |f|\,d\sigma. \]
Then Proposition~\ref{P:3MTDualProblem} is equivalent to the inequality
\begin{equation}\label{E:2L1}
\int_X|\alpha f+\beta g|\,d\sigma
\le \frac{|\alpha|+|\beta|+|\alpha+\beta|}{2}\max\Bigl\{\|f\|_{\sigma,1},\|g\|_{\sigma,1},\|f-g\|_{\sigma,1}\Bigr\}.
\end{equation}
We will say that a function is \textit{simple} if it only takes on a finite number of distinct values. By standard measure theory, there exist simple measurable real functions $f_m$, $g_m$ on $X$  such that $f_m\to f$ and $g_m\to g$ uniformly on $X$. Clearly $\| \alpha f_m + \beta g_m\|_{\sigma,1} \to \| \alpha f + \beta g\|_{\sigma,1}$. Likewise
$\|f_m\|_{\sigma,1}\to\|f\|_{\sigma,1}$ and $\|g_m\|_{\sigma,1}\to\|g\|_{\sigma,1}$ and $\|f_m-g_m\|_{\sigma,1}\to\|f-g\|_{\sigma,1}$. Thus, if the inequality
\eqref{E:2L1} holds for each pair of simple functions, then it holds for $f,g$.
So it suffices to establish \eqref{E:2L1} when $f,g$ are simple measurable real functions.

Hence, suppose that $f,g$ are simple measurable real functions on $X$.
We can write them as $f=\sum_{j=1}^n a_j1_{X_j}$ and $g=\sum_{j=1}^n b_j 1_{X_j}$,
where $\{ X_1,\dots,X_n \}$ is a measurable partition of $X$, and $a_j,b_j\in\RR$ for all $j$.
The inequality in \eqref{E:2L1} becomes
\[
\sum_{j=1}^n |\alpha a_j+\beta b_j|\sigma(X_j)
\le \frac{|\alpha|+|\beta|+|\alpha+\beta|}{2}\max\Bigl\{\sum_{j=1}^n|a_j|\sigma(X_j),\sum_{j=1}^n|b_j|\sigma(X_j),
\sum_{j=1}^n |a_j-b_j|\sigma(X_j)\Bigr\}.
\]
Writing 
\[x_j:=a_j\sigma(X_j), \quad \mathbf{x} = (x_1, \ldots, x_n)^T \in \RR^n\] and \[ y_j:=b_j\sigma(X_j), \quad \mathbf{y} = (y_1, \ldots, y_n)^T \in \RR^n\] we see that this becomes
\[
\|\alpha \mathbf{x}+\beta \mathbf{y}\|_1\le \frac{|\alpha|+|\beta|+|\alpha+\beta|}{2}\max\Bigl\{\|\mathbf{x}\|_1,\|\mathbf{y}\|_1,\|\mathbf{x-y}\|_1\Bigr\},
\]
where now $\mathbf{x},\mathbf{y}$ are vectors in $\RR^n$ and $\|\cdot\|_1$ denotes the usual $\ell^1$-norm on $\RR^n$ given by 
\begin{equation}
\label{E:ell1normDef} \|\mathbf{x}\|_1 := \sum_{j=1}^n |x_j|.
\end{equation} To summarize, to prove Proposition~\ref{P:3MTDualProblem} and consequently to prove Theorem~\ref{T:3MT}, it suffices to establish the following discrete result.

\begin{proposition}\label{P:3MTDiscreteFormulation}
Let $n\ge1$ and $\mathbf{x},\mathbf{y}\in\RR^n$. Then for all complex numbers $\alpha, \beta$ we have the inequality
\begin{equation} \label{E:3MTDiscreteFormulationEq}
\|\alpha \mathbf{x}+\beta \mathbf{y}\|_1\le \frac{|\alpha|+|\beta|+|\alpha+\beta|}{2}\max\Bigl\{\|\mathbf{x}\|_1,\|\mathbf{y}\|_1,\|\mathbf{x-y}\|_1\Bigr\}´.
\end{equation}
\end{proposition}

This reduction of the problem to the finite-dimensional setting allows us to use the tools of convex analysis.

\subsection{Optimization over a convex set}

Consider the set 
\begin{equation}\label{E:PolytopeDef}
C_n := \Bigl \{ (\mathbf{x},\mathbf{y}) \in \RR^n \times \RR^n : \|\mathbf{x}\|_1 \leq 1, \|\mathbf{y} \|_1 \leq 1, \| \mathbf{x - y}\|_1 \leq 1 \Bigr\}.
\end{equation}
Thus $C_n$ is a compact convex polytope in $\RR^n \times \RR^n$, and so it has a finite number of extreme points. That is, points of $C_n$ which do not lie in the interior of any line segment in $C_n$. A well-known theorem of Carathéodory says that each point of a compact convex polytope is a convex combination of its extreme points.

\begin{lemma} \label{L:ExtrPointIneqSufficiency} In order to establish Proposition~\ref{P:3MTDiscreteFormulation}, it suffices to show that the inequality \eqref{E:3MTDiscreteFormulationEq} holds for every extreme point of $C_n$.
\end{lemma}

\begin{proof}
Let us fix $\alpha, \beta \in \CC$ and $(\mathbf{x},\mathbf{y}) \in \RR^n \times \RR^n$. By the homogeneity of the inequality in \eqref{E:3MTDiscreteFormulationEq}, we may assume that
\begin{equation}
\label{E:NormalizationE1} \max  \Bigl \{ \|\mathbf{x}\|_1, \|\mathbf{y}\|_1, \|\mathbf{x-y}\|_1 \Bigr\}  = 1.\end{equation}
Then $(\mathbf{x},\mathbf{y}) \in C_n$ and so we may express it as a convex combination of the extreme points of $C_n$, namely
\[
\mathbf{x} = \sum_{k=1}^m t_k \mathbf{e^k}, \quad \mathbf{y} = \sum_{k=1}^m t_k \mathbf{f^k}
\] where $\mathbf{e^k} \in \RR^n$, $\mathbf{f^k} \in \RR^n$, the pairs $(\mathbf{e^k}, \mathbf{f^k})$ are extreme points of $C_n$, $t_k > 0$, and $\sum_{k=1}^m t_k = 1$. Note that since $(\mathbf{e^k}, \mathbf{f^k})$ is an extreme point of $C_n$, we must have 
\[ \max  \Bigl \{ \|\mathbf{e^k}\|_1, \|\mathbf{f^k}\|_1, \|\mathbf{e^k - f^k}\|_1 \Bigr\} = 1.\]
Since we are assuming that \eqref{E:3MTDiscreteFormulationEq} holds for extreme points, we can estimate
\begin{align*}
\| \alpha \mathbf{x} + \beta \mathbf{y}\|_1 &\leq \sum_{k=1}^m t_k \| \alpha \mathbf{e^k} + \beta \mathbf{f^k}\|_1 \\
&\leq  \sum_{k=1}^m t_j \frac{|\alpha| + |\beta| + |\alpha + \beta|}{2}  \max  \Bigl \{ \|\mathbf{e^k}\|_1, \|\mathbf{f^k}\|_1, \|\mathbf{e^k - f^k}\|_1 \Bigr\} \\
&= \frac{|\alpha| + |\beta| + |\alpha + \beta|}{2}\sum_{k=1}^m t_k \\
&= \frac{|\alpha| + |\beta| + |\alpha + \beta|}{2}.
\end{align*} Recalling our normalization in \eqref{E:NormalizationE1}, this is the desired estimate in \eqref{E:3MTDiscreteFormulationEq}.
\end{proof} 

From the above lemma and our sequence of reductions above, it follows that in order to prove Theorem~\ref{T:3MT} it suffices show that the inequality \eqref{E:3MTDiscreteFormulationEq} holds at every extreme point of the polytope $C_n$. Proposition~\ref{P:PolytopeExtrPoints} below characterizes these extreme points by partitioning them into three equivalence classes.

Note that $C_n$ is invariant under the following linear symmetries:
\begin{align}
&\begin{cases} \label{E:CSymmetry1}
x_j'&:= x_{\pi(j)} \\
y_j'&:= y_{\pi(j)} 
\end{cases}
\intertext{where $\pi$ is any permutation of $\{1,2,\dots,n\}$,}
&\begin{cases} \label{E:CSymmetry2} 
x_j'&:=\epsilon_j x_j \\
y_j'&:=\epsilon_j y_j 
\end{cases}
\intertext{for any choice of $\epsilon_1,\dots,\epsilon_n\in\{-1,1\}$, }
&\begin{cases} \label{E:CSymmetry3}
\mathbf{x'} &:=\mathbf{y} \\
\mathbf{y'} &:=\mathbf{x} 
\end{cases}
\intertext{and}
&\begin{cases} \label{E:CSymmetry4} 
\mathbf{x'}&:=\mathbf{x} \\
\mathbf{y'}&:=\mathbf{x-y}
\end{cases}
\end{align}
Denote by $G_n$ the group generated by these symmetries. As  these symmetries
are linear automorphisms of $\RR^n \times \RR^n$, it is clear that $G_n$ leaves invariant the set of extreme points $C_n$. We say that two extreme points of $C_n$ are $\mathit{G_n}$-\emph{equivalent} if there is an element of $G_n$ mapping one of them to the other. 
Thus the action of $G_n$ on $C_n$ partitions the set of extreme points of $C_n$ into a finite number of equivalence classes. Note that if the inequality \eqref{E:3MTDiscreteFormulationEq} holds for some $(\mathbf{x},\mathbf{y}) \in C_n$, then it holds also for any point of $C_n$ in the orbit of $(\mathbf{x},\mathbf{y})$ under the group action of $G_n$ on $C_n$.

The extreme points of $C_n$ are identified in the following proposition.

\begin{proposition}\label{P:PolytopeExtrPoints}
If $n\ge3$, then every extreme point $(\mathbf{x},\mathbf{y})$ of $C_n$ is $G_n$-equivalent to one of the pairs
\[
\begin{pmatrix}
1\\0\\0\\0\\ \vdots\\ 0
\end{pmatrix},
\begin{pmatrix}
1\\0\\0\\0\\ \vdots\\ 0
\end{pmatrix}
\quad\text{and}\quad
\begin{pmatrix}
1\\0\\0\\0\\ \vdots\\ 0
\end{pmatrix},
\begin{pmatrix}
1/2\\1/2\\0\\0\\ \vdots\\ 0
\end{pmatrix}
\quad\text{and}\quad
\begin{pmatrix}
1/2\\1/2\\0\\0\\ \vdots\\0
\end{pmatrix},
\begin{pmatrix}
1/2\\0\\1/2\\0\\ \vdots\\ 0
\end{pmatrix}.
\]
\end{proposition}

One can readily check that each of the three above pairs really is an extreme point of $C_n$. We omit the proof, since we do not actually need this fact. In the case that $n = 1$, the same result holds, but only the first kind of pair can arise. Likewise, if $n = 2$, the same result holds, but only the first two types of pairs can arise.

We will prove Proposition~\ref{P:PolytopeExtrPoints} in Section~\ref{S:ExtrPointSubsection}. For now let us see how Theorem~\ref{T:3MT} follows. In order to verify \eqref{E:3MTDiscreteFormulationEq} for all extreme points of $C_n$, it suffices to verify the inequality for the three pairs of vectors appearing in Proposition~\ref{P:PolytopeExtrPoints}. This is an easy task. For instance, if $(\mathbf{x},\mathbf{y})$ is the second pair in Proposition~\ref{P:PolytopeExtrPoints}, then we have
\begin{align*}
\|\alpha \mathbf{x} + \beta \mathbf{y}\|_1 &= |\alpha + \beta /2 | + |\beta|/2 
\\ &= |\alpha/2 + \alpha/2 + \beta/2| + |\beta|/2 
\\&\leq |\alpha + \beta|/2 + |\alpha|/2 + |\beta|/2.
\end{align*}
The inequality for the other two pairs is verified similarly. Then from Lemma~\ref{L:ExtrPointIneqSufficiency} we conclude that Proposition~\ref{P:3MTDiscreteFormulation} holds, from which Theorem~\ref{T:3MT} follows by the earlier reduction. 

It remains to prove Proposition~\ref{P:PolytopeExtrPoints}.

\subsection{Extreme points of the polytope}
\label{S:ExtrPointSubsection}

In the proof of Proposition~\ref{P:PolytopeExtrPoints}, the group $G_n$ generated by the symmetries \eqref{E:CSymmetry1}-\eqref{E:CSymmetry4} will be extensively used. In particular we will use the property that $(\mathbf{x},\mathbf{y})$ is an extreme point of $C_n$ if and only if some extreme point of $C_n$ is $G_n$-equivalent to it. Moreover, the following two observations will be useful to single out.

\begin{lemma} \label{L:ExtremePointsProofLem1}
If for a pair $(\mathbf{x},\mathbf{y}) \in C_n$ there exists two distinct indices $j,k$ such that 
\[
x_j > 0, \quad y_j > 0, \quad x_k > 0, \quad y_k > 0
\]
then $(\mathbf{x},\mathbf{y})$ is not an extreme point of $C_n$. 

More generally, if for two distinct indices $j, k$ we have that two of the quantities $x_jx_k$, $y_jy_k$ and $(x_j-y_j)(x_k-y_k)$ are non-zero and have the same sign, then $(\mathbf{x},\mathbf{y})$ is not an extreme point of $C_n$.
\end{lemma}

\begin{proof}
Using the symmetry \eqref{E:CSymmetry1} we may suppose that $j=1$, $k=2$. Note that $x_1 < 1$, $x_2 < 1$, since $\| \mathbf{x}\|_1 \leq 1$. The same is true for the corresponding coordinates of $\mathbf{y}$. Let $\mathbf{d} = (1, -1, 0, \ldots, 0)^T \in \RR^n$. It is easy to verify that if $t$ is a real number, and $|t|$ is sufficiently small, then we have 
\[ (\mathbf{x},\mathbf{y}) + t(\mathbf{d},\mathbf{d}) = (\mathbf{x}+t\mathbf{d}, \mathbf{y}+t\mathbf{d}) \in C_n. \] Thus $(\mathbf{x},\mathbf{y})$ lies on a line segment inside $C_n$, and so is not an extreme point of $C_n$.

The more general statement follows by applications of a sequence of symmetries in \eqref{E:CSymmetry1}-\eqref{E:CSymmetry4} to transform $(\mathbf{x},\mathbf{y})$ satisfying the more general assumption into a point $(\mathbf{x'}, \mathbf{y'})$ where the first two coordinates of the vectors $\mathbf{x'}$ and $\mathbf{y'}$ are positive.
\end{proof}

\begin{lemma} \label{L:ExtremePointsProofLem2}
If for a pair $(\mathbf{x},\mathbf{y}) \in C_n$ the vector $\mathbf{x}$ or $\mathbf{y}$ has at least three non-zero coordinates, then $(\mathbf{x},\mathbf{y})$ is not an extreme point of $C_n$. 
\end{lemma}

\begin{proof}
By using symmetries \eqref{E:CSymmetry1}-\eqref{E:CSymmetry3} we may suppose that coordinates $x_1, x_2, x_3$ are non-zero and positive. If two of the coordinates $y_1, y_2, y_3$ are positive, then by Lemma~\ref{L:ExtremePointsProofLem1} we conclude that $(\mathbf{x},\mathbf{y})$ is not an extreme point of $C_n$. In the contrary case, two of the coordinates $y_1, y_2, y_3$ are non-positive. Then again by Lemma~\ref{L:ExtremePointsProofLem1} and the symmetry \eqref{E:CSymmetry4} the pair $(\mathbf{x},\mathbf{x-y}) \in C_n$ is not extreme, and thus neither is $(\mathbf{x},\mathbf{y})$, since these two pairs are $G_n$-equivalent.
\end{proof}

We are ready to prove Proposition~\ref{P:PolytopeExtrPoints}. We denote by $\ell^1_n$ the space $\RR^n$ equipped with the norm $\| \cdot \|_1$ given by \eqref{E:ell1normDef}. Recall that the extreme points of the unit ball $B := \{ \mathbf{x} \in \RR^n : \|\mathbf{x}\|_1 \leq 1\}$ are the vectors with precisely one non-zero coordinate, this coordinate being equal to $\pm 1$. 

\begin{proof}[Proof of Proposition~\ref{P:PolytopeExtrPoints}]
We will split up the proof into three cases, each case corresponding to one of the pairs in the statement of the proposition.

\textbf{Case 1:} At least one of the norms $\|\mathbf{x}\|_1,\|\mathbf{y}\|_1,\|\mathbf{x-y}\|_1$ is strictly less than $1$. We will show that in this case $(\mathbf{x},\mathbf{y})$ is $G_n$-equivalent to the first pair in the statement of the proposition.

By applying a suitable combination of symmetries \eqref{E:CSymmetry1}-\eqref{E:CSymmetry4}, we may suppose that in fact $\|\mathbf{x-y}\|_1<1$. 
We claim that $\mathbf{x}$ must be an extreme point of the unit ball of $\ell^1_n$.
For if not, then it lies at the midpoint of a line segment $I$ such that
$\|\mathbf{x'}\|_1\le 1$ for all $\mathbf{x'}\in I$. Since $\|\mathbf{x-y}\|_1 < 1$, by shrinking $I$ if necessary, we also have $\|\mathbf{x'-y}\|_1<1$ for all $\mathbf{x'}\in I$.
Thus $I\times\{\mathbf{y}\}$ is a line segment in $C_n$ with interior point $(\mathbf{x},\mathbf{y})$, contradicting the fact that $(\mathbf{x},\mathbf{y})$ is extreme.

Likewise, $\mathbf{y}$ is extreme in the unit ball of $\ell^1_n$. 
Applying a suitable symmetry, we may suppose that $x_1=1$ and $y_j=\pm1$ for some $j$, 
all the other entries of $\mathbf{x}$ and $\mathbf{y}$ being $0$. Since we must have $\|\mathbf{x-y}\|_1<1$, this implies that actually
$j=1$ and $y_1=1$. Thus $(\mathbf{x},\mathbf{y})$ is equivalent to the first pair of vectors listed in the statement
of the proposition. This concludes Case~1.

\textbf{Case 2:} We have $\|\mathbf{x}\|_1 = \|\mathbf{y}\|_1 = \|\mathbf{x-y}\|_1 = 1$, and one of the vectors $\mathbf{x}$, $\mathbf{y}$ or $\mathbf{x-y}$ has only one non-zero coordinate. In this case, $(\mathbf{x},\mathbf{y})$ will be now shown to be $G_n$-equivalent to the second pair in the statement of the proposition.

Using our symmetries, we may suppose that $\mathbf{x} = (1, 0, \ldots, 0)^T$. Note that \[ \|\mathbf{x-y}\|_1 = |1 - y_1| + |y_2| + \ldots + |y_n| = 1\] and
\[ \|\mathbf{y} \|_1 = |y_1| + |y_2| + \ldots + |y_n| = 1\] force \[|1-y_1| = |y_1|,\] the unique real solution $y_1$ to this equation being $y_1 = 1/2$. By Lemma~\ref{L:ExtremePointsProofLem2}, $\mathbf{y}$ has only one other non-zero coordinate, and $\|\mathbf{y}\|_1 = 1$ forces this coordinate to be equal to $\pm 1/2$. Applying symmetries \eqref{E:CSymmetry1} and \eqref{E:CSymmetry2} we conclude that $(\mathbf{x},\mathbf{y})$ is $G_n$-equivalent to the second pair in the statement. This concludes Case~2.

\textbf{Case 3:} We have $\|\mathbf{x}\|_1 = \|\mathbf{y}\|_1 = \|\mathbf{x-y}\|_1 = 1$, and all of the vectors $\mathbf{x}$, $\mathbf{y}$ and $\mathbf{x-y}$ have exactly two non-zero coordinates. We will show that $(\mathbf{x},\mathbf{y})$ is $G_n$-equivalent to the third pair in the statement of the proposition.

This case is slightly more complicated than the previous two. As before, we may suppose that $x_1 > 0$ and $x_2 > 0$. We claim that $y_1$ and $y_2$ cannot both be equal to zero. If they were, then $\mathbf{x-y}$ has four non-zero coordinates, contrary to the assumption. In fact, precisely one of $y_1$ and $y_2$ must be non-zero. If both were non-zero, then since $\mathbf{x-y}$ has exactly two non-zero coordinates, we would have $x_1 - y_1 \neq 0$ and $x_2 - y_2 \neq 0$. Then the three quantities $x_1x_2$, $y_1y_2$ and $(x_1-y_1)(x_2 - y_2)$ would be non-zero, and Lemma~\ref{L:ExtremePointsProofLem1} would imply that $(\mathbf{x},\mathbf{y})$ is not an extreme point. 

By an application of symmetries we may, in addition to $x_1 > 0$ and $x_2 > 0$, suppose that $y_1 \neq 0$, $y_2 = 0$ and $y_3 = s > 0$. Since $x_1 + x_2 = 1$, we have $x_1 = t$, $x_2 = 1-t$ for some $t \in (0, 1)$. Our vectors thus have the following structure:

\[
x=
\begin{pmatrix}
t\\1-t\\0\\0\\ \vdots\\0
\end{pmatrix}, \quad
y=
\begin{pmatrix}
y_1 \\ 0 \\ s \\ 0\\\vdots\\0
\end{pmatrix}, \quad
x-y=
\begin{pmatrix}
t - y_1\\ 1-t \\-s\\0\\\vdots\\0
\end{pmatrix}.
\]

Recall that $\mathbf{x-y}$ has only two non-zero coordinates. Since $1-t \neq 0$ and $s \neq 0$, we conclude from the above that $t = y_1$. But then $\|\mathbf{x-y}\|_1 = 1-t + s = 1$, and so $t = s$. Finally, $1 = \|\mathbf{y}\|_1 = t + s = 2s$ shows that $s = t = 1/2$, and so $(\mathbf{x},\mathbf{y})$ is $G_n$-equivalent to the third pair in the statement of the proposition.
\end{proof}

\section{Proof of Theorem 1}
\label{S:ProofT1}

In addition to Theorem~\ref{T:3MT} from Section~\ref{S:3MSection}, we will also need some facts from plane geometry in order to prove Theorem~\ref{T:MainTheorem1}. In particular, we will need to discuss the \textit{minimum enclosing disk problem} appearing in computational geometry.

\subsection{Minimal enclosing disk}

Let $K$ be a compact subset of $\mathbb{C}$ containing at least two points. Among all closed disks which contain $K$ there exists a unique one of minimal radius. We will denote this disk by $\DD_K$ and call it the \textit{minimal disk} for $K$. The radius of $\DD_K$ will be denoted by $R(K)$. 
  
If $\DD_K$ is minimal for $K$, then the intersection $K \cap \partial \DD_K$ must obviously be non-empty. In fact, this intersection must contain at least two points, and there is also a restriction on the locations of the points in $K \cap \partial \DD_K$.

\begin{lemma} \label{L:arcLemma}
Let $K$ be a compact subset of $\mathbb{C}$ which contains at least two points. Then the intersection  $\partial \DD_K \cap K$ is not contained in any arc of $\partial \DD_K$ which has length strictly smaller than half of the circumference of $ \DD_K$. In particular, if $K \cap \partial \DD_K = \{a,b\}$ is a two-point set, then $a$ and $b$ are antipodal on $\partial \DD_K$
\end{lemma}

\begin{proof}
Seeking a contradiction, assume that $\partial \DD_K \cap K$ is contained in an arc of length strictly less than half of the circumference of $\DD_K$. By translation, rescaling and rotation of the setting, we may assume that $\DD_K$ is the unit disk, and that $\partial \DD_K \cap K$ is contained in some half-space \[\{ z \in \mathbb{C} : \Re z > \delta \}, \quad \delta > 0.\] By compactness, the distance between the compact sets $K$ and $\partial \DD_K \cap \{ z \in \mathbb{C} : \Re z \leq \delta/2 \}$ is positive. It follows that we may translate the disk $\DD_K$ in the positive direction of the real axis, and then shrink the radius of the translated disk slightly, and the resulting disk will still contain $K$, yet be of strictly smaller radius than $R_K$. See Figure \ref{fig:picture6}. This contradiction establishes Lemma~\ref{L:arcLemma}.
\end{proof}

\begin{figure}
    \centering    
\begin{tikzpicture}[scale=1.5]
    \filldraw[gray!20] (0.3, 0) circle (1.9);
    \draw[dashed] (0, 0) circle (2);
    %\draw [dashed] ({2*cos(80)}, {2*sin(80)}) -- ({2*cos(-80)}, {2*sin(-80)});
    %\draw[dashed] (0,2) -- (0,-2);
    %\draw[->] ({2*cos(80)}, 0) -- (({2*cos(80) + 0.3}, 0);
    \draw [black,ultra thick,domain=-45:45] plot ({2*cos(\x)}, {2*sin(\x)});
\end{tikzpicture}
     \caption{The initial disk $\DD_K$ is the dashed circle, and we assume that $\partial \DD_K \cap K$ is contained in the black thick arc. Then $K$ will be contained in the grey disk which is obtained from $\DD_K$ by first translating $\DD_K$ in the direction of the positive real axis, and then slightly shrinking the translated disk. This contradicts the minimality of $\DD_K$.}
    \label{fig:picture6}
\end{figure}
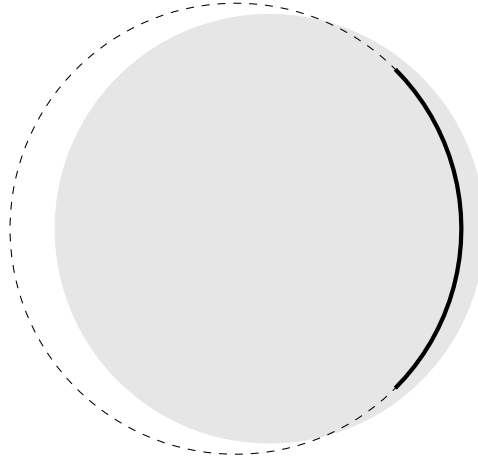

\begin{lemma} \label{L:threePointCor} Let $T = \{a, b, c\}$ be a three-point set. If $D$ is a closed disk for which $T \subset \partial D$, and $T$ is not contained in any arc of $\partial D$ which is strictly smaller than half of the circumference of $D$, then $D = \DD_T$.
\end{lemma}

\begin{proof}
Assume, seeking a contradiction, that $D \neq \DD_T$, and so that $R(T)$ is strictly smaller than the radius of $D$. Since $\partial D$ is the unique circle passing through the three points $a,b,c$, we must have that $T \cap \partial \DD_T$ contains precisely two points. Say $a, b \in \partial \DD_T$ but $c \not\in \partial \DD_T$. Lemma~\ref{L:arcLemma} implies that $a$ and $b$ are antipodal on $\DD_T$. By translation, rescaling and rotation, we may assume that $\DD_T$ is the unit disk, $a = i, b = -i$, $c$ has non-negative real part and $|c| < 1$. After these operations, we have that $R(T) = 1$ and the circumference of $D$ is larger than $2 \pi$. Thus by hypothesis, surely $T$ is not contained in any arc of $\partial D$ of length strictly smaller than $\pi$. But the shorter of the arcs of $\partial D$ which contains $T$ is then contained in $\{ z \in \mathbb{C} : 0 \leq \Re z, |z| \leq 1\}$, and so this arc must have a length smaller than $\pi$. This is a contradiction, and the lemma follows.
\end{proof}

\subsection{Reduction to three-point sets}

The following simple result on minimal disks makes it possible to apply Theorem~\ref{T:3MT} to more than three measures.

\begin{lemma} \label{L:3PointSupportingSubsetLemma}
Let $K$ be a compact subset of $\CC$ containing at least two points. There exists a subset $T \subset K$ which contains at most three points and for which $\DD_K = \DD_T$. In particular, $R(K) = R(T)$. 
\end{lemma}

It may be convenient to refer to Figure~\ref{fig:fig4} during the reading of the proof.

\begin{proof} 
If there are two points in $K$ which are antipodal on $\partial \DD_K$, then we take $T$ to consist of those two points. Clearly $\DD_K = \DD_T$. 
In the case that no pair of antipodal points of $\partial \DD_K$ are contained in $K$, let $J$ be the shortest closed arc of $\partial \DD_K$ which contains $K$, and let $a,b \in J \cap K$ be the end-points of $J$. By Lemma~\ref{L:arcLemma}, the length of $J$ is strictly larger than half of the circumference of $\partial \DD_K$, and so $J$ is the longer of the two arcs between $a$ and $b$. Let $\tilde{a}$ and $\tilde{b}$ be points on $\partial \DD_K$ which are antipodal to $a$ and $b$, respectively. By assumption, $\tilde{a} \not\in K, \tilde{b} \not\in K$. We claim that the shorter of the two open arcs between $\tilde{a}$ and $\tilde{b}$ must contain points of $K$. If not, then the longer of the two arcs between $\tilde{a}$ and $\tilde{b}$ would contain $K$ in its interior, and this arc has the same length as $J$. A routine compactness argument would lead to a contradiction to the minimality of $J$. 

Let $T = \{a, b, c\}$, where $c \in K$ is any point contained in the shorter open arc between $\tilde{a}$ and $\tilde{b}$. Note that any arc containing $T$ must contain either $\tilde{a}$ or $\tilde{b}$. Then such an arc contains two antipodal points on $\DD_K$, and so it has a length which is at least half of the circumference of $\DD_K$. By Lemma~\ref{L:threePointCor} we conclude that $\DD_K = \DD_T$. 
\end{proof}

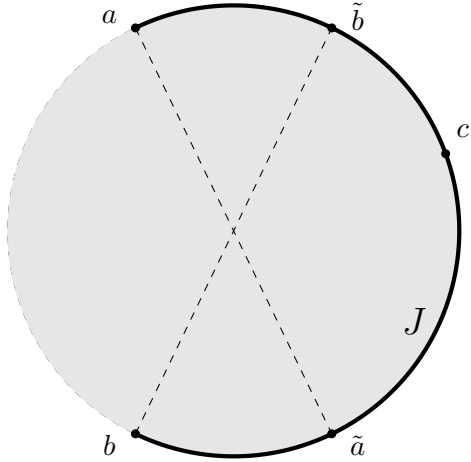
\begin{figure}
\centering
\begin{tikzpicture}[scale=3]
    \draw[dashed] (0,0) circle (1);
    \filldraw[gray!20] (0, 0) circle (1);
    
    \node at (-0.55, 0.95) {$a$};
    \filldraw (-0.435, 0.9) circle (0.5pt);  
    \node at (0.55, -0.95) {$\tilde{a}$};
    \filldraw (0.435, -0.9) circle (0.5pt);     
    \draw[dashed] (-0.435, 0.9) -- (0.435, -0.9);
    
    \node at (-0.55, -0.95) {$b$};
    \filldraw (-0.435, -0.9) circle (0.5pt);     
    \node at (0.55, 0.95) {$\tilde{b}$};
    \filldraw (0.435, 0.9) circle (0.5pt);     
    \draw[dashed] (-0.435, -0.9) -- (0.435, 0.9);
    1
    \node at ({cos(20)+0.075}, {sin(20)+0.1}) {$c$};
    \filldraw ({cos(20)}, {sin(20)}) circle (0.5pt);  
    
    \draw [black,ultra thick,domain=-115:115, samples=1000] plot ({cos(\x)}, {sin(\x)});
    \node at (0.8, -0.4) {\Large $J$};
\end{tikzpicture}

\caption{The thick arc $J$ between $a$ and $b$ is the smallest containing the compact set $K$. It follows that the shorter arc between the antipodal points $\tilde{a}$ an $\tilde{b}$ must contain points of $K$.}
\label{fig:fig4}
\end{figure}

\subsection{Finalizing the proof}
We are finally ready to give a proof of the equality $c_\RR(\Omega) = c_\CC(\Omega)$.
\begin{proof}[Proof of Theorem~\ref{T:MainTheorem1}] Since $c_\RR(\Omega) \leq c_\CC(\Omega)$, it will suffice to show the reverse inequality. To this end, we need to show that given $f \in C(\partial \Omega)$ satisfying $\|f \|_{\partial \Omega} \leq 1$, we have that $\|K_\Omega f + \CC \mathbf{1}\|_{\partial \Omega} \leq c_\RR(\Omega)$. Since $K_\Omega f$ is continuous, the image $K = K_\Omega f(\partial \Omega)$ is a compact subset of $\CC$. If $K$ consists of a single point, then $\|K_\Omega f + \CC \mathbf{1}\|_{\partial \Omega} = 0$, and the proof is complete. In other case, let $\DD_K$ be the minimal disk for $K$. We use Lemma~\ref{L:3PointSupportingSubsetLemma} to obtain a three-point set $T = \{a,b,c\} \subset K$ for which $R(T) = R(K)$ (note that if $K \cap \partial \DD_K$ contains only two points $\{a,b\}$, then we may pick $c \in K$ arbitrarily to complete $T$ to a three-point set). The geometric interpretation of the quotient norm in $C(\partial \Omega)/\CC \mathbf{1}$ implies that $\|K_\Omega f + \CC \mathbf{1}\|_{\partial \Omega} = R(K) = R(T)$. Since $T$ is contained in the image of $K_\Omega f$, there exists $\zeta_1, \zeta_2, \zeta_3 \in \partial \Omega$ such that
\[ (a,b,c) = \bigl( K_\Omega f(\zeta_1) , K_\Omega f(\zeta_2), K_\Omega f(\zeta_3)\bigr).\] Since $K_\Omega f(\zeta_j) = \int_{\partial \Omega} f \, d\mu_{\zeta_j}$, we may apply Theorem~\ref{T:3MT} to $X = \partial \Omega$, $\mu_j = \mu_{\zeta_j}$ for $j = 1,2,3$, and conclude that the operator $\LL: C(\partial \Omega) \to \CC^3 / \CC \mathbf{1}$ defined by
\[ \LL : f \mapsto \bigl( K_\Omega f(\zeta_1) , K_\Omega f(\zeta_2), K_\Omega f(\zeta_3)\bigr) + \CC \mathbf{1}\] has a norm satisfying the bound \eqref{E:3MTEstimate}. With $\| \cdot \|_\infty$ denoting the norm on $\CC^3 / \CC \mathbf{1}$ given in \eqref{E:C3supnorm}, we obtain
\begin{align*}
\|K_\Omega f + \CC \mathbf{1}\|_{\partial \Omega} = R(T) &= \|(a,b,c) + \CC \mathbf{1}\|_\infty \\
&= \|\LL f \|_{\infty} \\
&\leq \|\LL\|_{C(\partial \Omega) \to \CC^3 / \CC \mathbf{1}} \\
&\leq \frac{1}{2} \max_{j,k} \|\mu_{\zeta_j} - \mu_{\zeta_k}\| \\
&\leq \frac{1}{2} \sup_{\zeta, \zeta' \in \partial \Omega} \|\mu_\zeta - \mu_{\zeta'}\|\\
& = c_\RR(\Omega).
\end{align*}
\end{proof}

The earlier mentioned extension of Theorem~\ref{T:3MT} to an \textit{n-measures theorem} is obtained by employing the same argument as in the above proof. The normed space $\CC^n / \CC \mathbf{1}$ appearing below is defined analogously to the case $n=3$ treated in Section~\ref{S:Section2Definitions}.

\begin{theorem} \label{T:NMT} Let $C(X)$ be the space of continuous functions on a compact Hausdorff space $X$, $n \geq 3$ an integer, and $\LL: C(X) \to \CC^n / \CC \mathbf{1}$ the operator defined by 
\[
\LL f = \big( \mu_1(f), \ldots, \mu_n(f) \big) + \CC \mathbf{1}
\]
where $\mu_1, \ldots, \mu_n$ are finite real-valued Borel measures on $X$. Then
\[ \| \LL \|_{C(X) \to \CC^n / \CC \mathbf{1}} = \frac{1}{2} \max_{j,k} \|\mu_j - \mu_k\|.\]
\end{theorem}

\begin{proof}
We use Lemma~\ref{L:3PointSupportingSubsetLemma} to pick a three-point subset $T$ of $K = \{ \mu_j(f) \}_{j=1}^n$ for which we have $R(K) = R(T)$, and apply Theorem~\ref{T:3MT} as in the preceeding proof.
\end{proof}

\section{Proof of Theorem 2}
\label{S:ProofT2}

\subsection{Exploiting subsequences} We will argue by contradiction in order to prove Theorem~\ref{T:Maintheorem2}. That is, we will assume that there exists a convex domain $\Omega$ with $a(\Omega) = 1$, and so that there exists a sequence of functions $(f_n)$ in $\A(\Omega)$ which satisfy 
\[ \|f_n + \CC \mathbf{1}\|_{\Omega} = 1
\]
and
\begin{equation} \label{E:Thm2ProofInitAssmpt}
\lim_{n \to \infty} \| K_\Omega f_n + \CC \mathbf{1}\|_{\Omega} = 1.
\end{equation}
We shall see that this leads to a contradiction. The proof technique below is different from the one employed by Schober in \cite{Sch68} in his proof of Neumann's lemma, and analyticity is used only at the very end of the proof. In fact, we shall remark at the end of the section how our arguments lead to a new proof of Neumann's lemma which is different from the one in \cite{Sch68}.

Thus, for now, we assume merely that $f_n \in C(\partial \Omega)$, and we will derive certain consequences of \eqref{E:Thm2ProofInitAssmpt}. In the course of the proof we shall replace the sequence $(f_n)$ by a subsequence multiple times, and for convenience we will not be changing the subscripts. We may suppose that $\|f_n\|_\Omega = 1$, and consequently that the images
\[ K_\Omega f_n(\partial \Omega) := \{ K_\Omega f_n(\zeta) : \zeta \in \partial \Omega \}
\] are contained in a closed disk of radius $1$ centred at the origin. For large $n$, this observation and \eqref{E:Thm2ProofInitAssmpt} forces there to be points of the image of $K_\Omega f_n$ outside of any disk centred at the origin of radius strictly less than $1$. By exchanging $f_n$ for a unimodular multiple of itself, we may thus assume that there exists a sequence of points $(\zeta_n)$ in $\partial \Omega$ for which we have 
\begin{equation}\label{E:Thm2ProofEq1}
\lim_{n \to \infty} K_\Omega f_n(\zeta_n) = \lim_{n \to \infty} \int_{\partial \Omega} f_n \, d\mu_{\zeta_n} = 1.
\end{equation}
Using that the functions $f_n$ are bounded by $1$ in modulus, and the positive measures $d\mu_\zeta$ are of unit mass, we obtain 
\begin{align*}
\lim_{n \to \infty} \int_{\partial \Omega} |f_n - 1|^2 d\mu_{\zeta_n} &= \lim_{n \to \infty} \int_{\partial \Omega} \big(|f_n|^2 - 2 \Re f_n + 1 \big) d\mu_{\zeta_n} \\
&\leq \lim_{n \to \infty} \Big( 2 - 2\Re \int_{\partial \Omega} f_n \, d\mu_{\zeta_n} \Big) = 0. 
\end{align*} Recall from \eqref{E:MuZetaEq} that $\rho_{\zeta_n}$ denotes the $ds$-absolutely continuous part of $\mu_{\zeta_n}$. The above computation implies that
\begin{equation} \label{E:Thm2ProofEq2}
\lim_{n \to \infty} \int_{\partial \Omega} |f_n - 1|^2 \rho_{\zeta_n}ds = 0.
\end{equation}
Compactness of the boundary $\partial \Omega$ implies that we may assume convergence of the sequence $(\zeta_n)$ to some points $\zeta \in \partial \Omega$. The following lemma shows that we may replace in \eqref{E:Thm2ProofEq2} the densities $\rho_{\zeta_n}$ with the density $\rho_{\zeta}$.

\begin{lemma} \label{L:Thm2ProofLemma1} With notations as above, we have 
\begin{equation} \label{E:Thm2ProofEq3}
\lim_{n \to \infty} \int_{\partial \Omega} |f_n - 1|^2 \rho_{\zeta}ds = 0.
\end{equation} Consequently, after passing to a subsequence, we can ensure that \[ \lim_{n \to \infty} f_n(\sigma) = 1\] for almost every $\sigma \in \partial \Omega$ with respect to the measure $\rho_\zeta d s$.
\end{lemma}

\begin{proof} Note that whenever $\sigma$ is not a corner of $\partial \Omega$ or any of the points $\zeta_n$ or $\zeta$, we have \[
    \rho_{\zeta_n}(\sigma) - \rho_\zeta(\sigma) =  \Re \frac{(\zeta_n - \zeta) N(\sigma)}{\pi(\sigma-\zeta)(\sigma-\zeta_n)}.\] If $B = B(\zeta, \delta)$ is a disk around $\zeta$ of small radius $\delta > 0$, then for large enough $n$ the denominator on the right-hand side above is uniformly bounded from below for $\sigma \in \partial \Omega \setminus B$, with exception of a countable set. This shows uniform convergence of $\rho_{\zeta_n}(\sigma)$ to $\rho_\zeta(\sigma)$ for $\sigma \in \partial \Omega \setminus B$, again with exception of an at most countable set. Since $|f_n - 1|^2 \leq 4$, we obtain from \eqref{E:Thm2ProofEq2} that
\begin{align*}
\limsup_{n \to \infty} \int_{\partial \Omega} |f_n - 1|^2 \rho_{\zeta}ds & \leq \limsup_{n \to \infty} \int_{\partial \Omega \cap B} |f_n - 1|^2 \rho_{\zeta}ds
\\
&+ \limsup_{n \to \infty} \int_{\partial \Omega \setminus B} |f_n - 1|^2 \rho_{\zeta}ds
\\
&\leq  4 \int_{\partial \Omega \cap B} \rho_\zeta ds.
\end{align*} Since $\partial \Omega \cap B$ is an arc of length which tends to $0$ as the radius $\delta$ of $B$ tends to $0$, the last quantity above can be made arbitrarily small by choosing $\delta$ small enough. This establishes \eqref{E:Thm2ProofEq3}. Basic measure theory now implies that we may pass again to a subsequence and ensure the pointwise convergence $f_n \to 1$ almost everywhere with respect to $\rho_\zeta d s$.
\end{proof}

%Since the measure $\rho_\zeta$ is non-zero and absolutely continuous with respect to arclength $s$, there exists a Borel set $E \subset \partial \Omega$ which has positive arclength measure, and on which the sequence $\{f_n\}_{n \geq 1}$ converges pointwise to $1$.

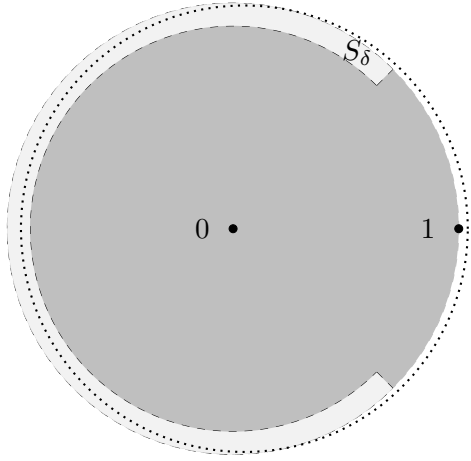
\begin{figure}
\centering
\begin{tikzpicture}[scale=3]
    \draw [dashed, name path=A, domain=45:315, samples=1000] plot ({cos(\x)}, {sin(\x)});
    \filldraw[gray!50, dashed] (0, 0) circle (1);
    \draw [dashed, name path=B, domain=45:315, samples=1000] plot ({0.9*cos(\x)}, {0.9*sin(\x)});
    \draw[dashed] (0.9*0.707, 0.9*0.707) -- (0.707, 0.707);
    \draw[dashed] (0.9*0.707, -0.9*0.707) -- (0.707, -0.707);

    \tikzfillbetween[of=A and B]{gray!10};
    \draw[dotted, thick, opacity= 2] (0.05, 0) circle (0.99);

    \filldraw (0, 0) circle (0.5pt) node[left, xshift = -5pt]{$0$};     
    
    \filldraw (1, 0) circle (0.5pt) node[left, xshift= - 5pt]{$1$}; 

    \node at (0.55, 0.780) {$S_\delta$};

\end{tikzpicture}
\caption{The unit disk in dark grey with the strip $S_\delta$ removed. The dotted circle containing the dark grey area has a radius slightly smaller than $1$.}
\label{fig:fig2}
\end{figure}

Out next observation extracts more information from \eqref{E:Thm2ProofInitAssmpt}. Consider the strips 
\[S_\delta = \{ z = re^{it} : 1-\delta < r < 1, |t| \in [\pi/4, \pi] \}, \quad \delta > 0. \] These strips have a fixed large ``length" but shrinking ``width". One such strip is marked in Figure \ref{fig:fig2}. We claim that each one of the strips $S_\delta$ intersects the images $K_\Omega f_n(\Omega)$ non-trivially for infinitely many indices $n$. For if not, then for some fixed $\delta > 0$, we would have that $S_\delta \cap K_\Omega f_n(\partial \Omega)= \emptyset$ for all sufficiently large $n$, which means that the images $K_\Omega f_n(\partial \Omega)$ are entirely contained in $B(0,1) \setminus S_\delta$, where $B(0,1)$ denotes the closed disk of radius $1$ centred at the origin. But if $\epsilon_1$ and $\epsilon_2$ are sufficiently small positive numbers, then $B(0,1) \setminus S_\delta \subset B(\epsilon_1, 1-\epsilon_2)$, a disk of radius $1-\epsilon_2$ centred at the point $\epsilon_1 \in \RR$. See Figure~\ref{fig:fig2}. Recalling the geometric interpretation of the norm $\| K_\Omega f_n + \CC\mathbf{1}\|_{\partial \Omega}$ as the radius of the smallest disk containing the image of $K_\Omega f_n$, we would arrive at a contradiction to \eqref{E:Thm2ProofInitAssmpt}. Thus every strip $S_\delta$ contains points in the image of $K_\Omega f_n$ for infinitely many $n$. 

\begin{lemma} \label{L:Thm2ProofLemma2} With notations as above, we may pass to a subsequence again, and obtain a new sequence $(\zeta_n')$ which converges to some point $\zeta' \in \partial \Omega$, and such that \[ \lim_{n \to \infty} f_n(\sigma) = \alpha\] for some unimodular constant $\alpha \neq 1$ and for almost every $\sigma \in \partial \Omega$ with respect to the measure $\rho_{\zeta'} d s$.
\end{lemma}

\begin{proof}
Since each strip $S_\delta$ intersects the images of $K_\Omega f_n$ for infinitely many $n$, passing to a subsequence and a routine compactness argument produces a sequence $(\zeta'_n)$ convergent to some $\zeta' \in \partial \Omega$, for which $K_\Omega f_n(\zeta'_n) \to \alpha$, with $\alpha$ unimodular and lying in the closure of each of the strips $S_\delta$. Thus $\alpha \neq 1$. We therefore merely need to repeat the previous arguments to see that, after passing to a subsequence, we will have $f_n(\sigma) \to \alpha$ for almost every $\sigma$ with respect to the measure $\rho_{\zeta'} d s$.
\end{proof}

\subsection{Proof of Theorem~\ref{T:Maintheorem2}} The above arguments are valid for $f_n \in C(\partial \Omega)$. However, under the assumption of analyticity, the sequence $(f_n)$ cannot converge to two different constants on two different sets of positive arclength measure. To make this statement precise, we appeal to the classical theory of analytic functions in the (open) unit disk $\DD = \{ z \in \CC : |z| < 1 \}$. Here \cite[Chapter II]{Ga06} is an excellent reference for the claims made in the following proof.

\begin{proof}[Proof of Theorem~\ref{T:Maintheorem2}]  Let $H^\infty = H^\infty(\DD)$ be the space of bounded analytic functions in $\DD$, identified as usual through boundary function correspondence with a weak-star closed subspace of the space $L^\infty(\partial \DD) = (L^1(\partial \DD)^*$ of bounded measurable functions on $\partial \DD$, the dual of the Lebesgue space $L^1(\partial \DD)$ of functions integrable on $\partial \DD$ with respect to the Lebesgue measure (arclength measure) on $\partial \DD$. It is well known that a function $\tilde{f} \in H^\infty$ which vanishes on a subset of positive Lebesgue measure on $\partial \DD$ must vanish identically. 

Fix some conformal mapping $\phi: \DD \to \Omega$. Under the assumption that $f_n \in \A(\Omega)$, $\|f_n\|_\Omega \leq 1$, the functions
\[
\tilde{f_n} := f_n \circ \phi \in H^\infty, \quad n \geq 1
\] are bounded in modulus by $1$ in $\DD$. By Carath\'eodory's classical theorem (see, for instance, \cite[Chapter I.3]{GM05}), $\phi$ extends to a homeomorphism between $\partial \DD$ and $\partial \Omega$. If $\|K_\Omega f_n + \CC \mathbf{1}\|_\Omega \to 1$, then Lemmas~\ref{L:Thm2ProofLemma1} and \ref{L:Thm2ProofLemma2} show that there exist two sets $E, E' \subset \partial \Omega$ which have positive arclength measure, such that
\[
\lim_{n \to \infty} \tilde{f_n}(\lambda) = 1, \quad \lambda \in \phi^{-1}(E)
\]
and
\[
\lim_{n \to \infty} \tilde{f_n}(\lambda) = \alpha, \quad \lambda \in \phi^{-1}(E').
\]
Since $\Omega$ is convex, the curve $\partial \Omega$ is rectifiable, and general theory of harmonic measures tells us that the sets $\phi^{-1}(E)$ and $\phi^{-1}(E')$ have positive Lebesgue measure (see \cite[Chapter VI]{GM05}). Since $L^1(\partial \DD)$ is separable and the functions $\tilde{f_n}$ are uniformly bounded by $1$ in modulus, the usual Helly-type selection process will produce a subsequence of $(\tilde{f}_n)$ which converges in the weak-star topology to some function $\widetilde{f} \in H^\infty$. By the above pointwise convergence, we must have $\tilde{f} \equiv 1$ on $\phi^{-1}(E)$ and $\tilde{f} \equiv \alpha $ on $\phi^{-1}(E')$. Then the non-zero function $\widetilde{f} - 1$ vanishes on the subset $\phi^{-1}(E)$ of positive Lebesgue measure on $\partial \DD$. This is a contradiction, which shows that our assumption $\|K_\Omega f_n + \CC \mathbf{1}\|_\Omega \to 1$ must be false. Theorem~\ref{T:Maintheorem2} follows.
\end{proof}

\subsection{A proof of Neumann's lemma}
We indicate how one may proceed to use our above arguments to obtain a proof of Neumann's lemma, stating that $c(\Omega) = 1$ if and only if $\Omega$ is a triangle or a quadrilateral. We need only the following simple geometric observation regarding the densities $\rho_\zeta$.

\begin{lemma} \label{L:NeumannLemmaLemma} Fix $\zeta \in \partial \Omega$. Any $\sigma \in \partial \Omega \setminus \{\zeta\}$ which is not a corner of $\partial \Omega$ and which satisfies $\rho_\zeta(\sigma) = 0$ is contained in the union of at most two line segments of $\partial \Omega$ containing $\zeta$.
\end{lemma}

\begin{proof} It will suffice to show that all $\sigma$ satisfying the above conditions are contained in at most two different tangent lines to $\Omega$. To see this, recall formula \eqref{E:DensityCircleRadiusFormula}. The condition $\rho_\zeta(\sigma) = (2 \pi R_{\zeta,\sigma})^{-1} = 0$ gives $R_{\zeta, \sigma} = \infty$, and so $\zeta$ is contained in the tangent line to $\Omega$ at $\sigma$. The tangent line divides the plane $\CC$ into two half-planes, one of which contains $\Omega$. Assume that two different tangent lines, at $\sigma$ and $\sigma'$, intersect at $\zeta$. They divide the plane $\CC$ into four sectors, and by convexity precisely one of those sectors contains $\Omega$. Now, any line which passes through $\zeta$ and the open sector containing $\Omega$ must separate $\sigma, \sigma' \in \partial \Omega$. Therefore, it is not a tangent to $\Omega$.
\end{proof}

Neumann's lemma is established as follows. Assume that $c(\Omega) = 1$. From Lemmas~\ref{L:Thm2ProofLemma1} and \ref{L:Thm2ProofLemma2} we see that two points $\zeta, \zeta'$ exist for which the measures $\rho_\zeta ds$ and $\rho_{\zeta'} ds$ are mutually singular. From Lemma~\ref{L:NeumannLemmaLemma} we deduce that the support of $\rho_\zeta ds$ is the union of at most two line segments containing $\zeta'$, and the complement of the support of $\rho_\zeta ds$ is also a union of at most two line segments. Thus $\partial \Omega$ is the union of at most four line segments.

\section{Examples}
\label{S:ExamplesSection}

In this section, we compute and estimate the configuration constants for some types of domains. 

\subsection{Configuration constant of an ellipse}

For $a,b>0$, let
\[
\Omega_{a,b}:=\Bigl\{x+iy\in\CC: \frac{x^2}{a^2}+\frac{y^2}{b^2}\le1\Bigr\}
\] be the ellipse centred at the origin with semi-axes of lengths $a$ and $b$, respectively. It is quite remarkable that the configuration constant can in this case be computed explicitly.

\begin{proposition}\label{P:EllipseConst}
With the above notation, we have
\[
c(\Omega_{a,b})=\frac{2}{\pi}\arctan\Bigl(\frac{1}{2}\Bigl|\frac{b}{a}-\frac{a}{b}\Bigr|\Bigr).
\]
\end{proposition}

In order to prove the proposition, our first step is to derive an expression for the density
of the Neumann-Poincaré kernel of $\Omega_{a,b}$. The boundary $\partial\Omega_{a,b}$ is parametrized by
\begin{equation}\label{E:param}
\gamma(t):=a\cos t+ib\sin t, \quad t\in[0,2\pi].
\end{equation}
Here $[0,2\pi]$ can be replaced by any interval of length $2\pi$. Recalling formula \eqref{E:MuZetaDensityFormula} for $\mu_\zeta$ and setting $\zeta = \gamma(s)$, $\sigma = \gamma(t)$, we obtain
\begin{align}
d \mu_{\gamma(s)}(\gamma(t)) &= \rho_{\gamma(s)}(\gamma(t)) \, ds(\gamma(t)) \label{E:MuDensityParam} \\
&= \frac{1}{\pi} \Im \Bigg(\frac{T(\gamma(t))}{\gamma(t) - \gamma(s)}  \Bigg) |\gamma'(t)| \, dt \nonumber \\
&= \frac{1}{\pi}\Im\frac{\gamma'(t)}{\gamma(t)-\gamma(s)} \, dt. \nonumber
\end{align} 
Using \eqref{E:param}, this formula can be greatly simplified.

\begin{lemma}\label{L:density}
With the notation above, we have
\begin{equation}\label{E:density}
d\mu_{\gamma(s)}(\gamma(t))=\frac{1}{2\pi}\frac{A}{1+B\cos(t+s)}\,dt,
\quad s,t\in[0,2\pi],
\end{equation}
where
\[
A:=\frac{2ab}{a^2+b^2}
\quad\text{and}\quad
B:=\frac{b^2-a^2}{b^2+a^2}.
\]
\end{lemma}

The lemma is established by combining \eqref{E:param} and \eqref{E:MuDensityParam}, and then using elementary trigonometric identities to simplify the resulting expression.  

%\begin{proof}
%A combination of \eqref{E:param} and \eqref{E:MuDensityParam} leads to
%\[
%d\mu_{\gamma(s)}(\gamma(t))=\frac{1}{\pi}
%\frac{ab\cos^2 t+ab\sin^2t-ab\cos t\cos s-ab\sin t\sin s}{a^2(\cos t-\cos s)^2+b^2(\sin t-\sin s)^2}.
%\]
%Thanks to the addition and
%subtraction formulas for $\cos,\sin$, namely
%\begin{align*}
%1-\cos(t-s)&=2\sin^2\Bigl(\frac{t-s}{2}\Bigr),\\
%\cos t-\cos s&=2\sin\Bigl(\frac{t-s}{2}\Bigr)\sin\Bigl(\frac{t+s}{2}\Bigr),\\\sin t-\sin s&=2\sin\Bigl(\frac{t-s}{2}\Bigr)\cos\Bigl(\frac{t+s}{2}\Bigr),
%\end{align*}
%this last expression simplifies to
%\[
%d\mu_{\gamma(s)}(\gamma(t))=
%\frac{1}{\pi}\frac{ab}{2a^2\sin^2((t+s)/2)+2b^2\cos^2((t+s)/2)}.
%\]
%Finally, noting that
%\begin{align*}
%2\sin^2\Bigl(\frac{t+s}{2}\Bigr)=1-\cos(t+s),\\
%2\cos^2\Bigl(\frac{t+s}{2}\Bigr)=1+\cos(t+s),
%\end{align*}
%we obtain the formula \eqref{E:density}.
%\end{proof}

With this formula in hand, we  now evaluate the configuration constant of the ellipse $\Omega_{a,b}$.

\begin{proof}[Proof of Proposition~\ref{P:EllipseConst}]
Using the formulas \eqref{E:RconfigConstantFormula} and \eqref{E:density}, we obtain
\[
c(\Omega_{a,b})
=\sup_{s_1,s_2\in[0,2\pi]}\frac{1}{2}\frac{1}{2\pi}\int_{[-\pi, \pi]}\Bigl|\frac{A}{1+B\cos(t+s_1)}-\frac{A}{1+B\cos(t+s_2)}\Bigr|\,dt.
\]
By the periodicity of $\cos$, this last expression simplifies to
\[
c(\Omega_{a,b})
=\sup_{s\in(0,2\pi)}\frac{1}{2}\frac{1}{2\pi}\int_{[-\pi,\pi]}\Bigl|\frac{A}{1+B\cos(t+s)}-\frac{A}{1+B\cos(t)}\Bigr|\,dt.
\]
For the time being, let us assume that $b\ge a$, so $B\ge0$.
Using the fact that \eqref{E:density} is the density of a probability
measure for each $s\in[0,2\pi]$, we have
\begin{align*}
&\frac{1}{2}\frac{1}{2\pi}\int_{[-\pi,\pi]}\Bigl|\frac{A}{1+B\cos(t+s)}-\frac{A}{1+B\cos(t)}\Bigr|\,dt\\
&=\frac{1}{2\pi}\int_{\{t:\cos t\ge\cos(t+s)\}}\Bigl(\frac{A}{1+B\cos(t+s)}-\frac{A}{1+B\cos(t)}\Bigr)\,dt.
\end{align*}
We readily verify that $\cos(t) \geq \cos(t+s)$ if and only if $t\in[-s/2,\,\pi-s/2]$. Therefore 

%Since
%\[
%\cos(t+s)-\cos t=2\sin(s/2)\sin(t+s/2),
%\]
%and since $\sin(s/2)>0$ for all $s\in(0,2\pi)$, we have
%$\cos t\le \cos(t+s)$ if and only if $\sin(t+s/2)\ge0$, that is to say,
%if and only if $t\in[-s/2,\,\pi-s/2]$. Therefore

\begin{align*}
&\frac{1}{2}\frac{1}{2\pi}\int_{[-\pi,\pi]}\Bigl|\frac{A}{1+B\cos(t+s)}-\frac{A}{1+B\cos(t)}\Bigr|\,dt\\
&=\frac{1}{2\pi}\int_{-s/2}^{\pi-s/2}\Bigl(\frac{A}{1+B\cos(t+s)}-\frac{A}{1+B\cos(t)}\Bigr)\,dt\\
&=\frac{1}{2\pi}\int_{s/2}^{\pi+s/2}\frac{A}{1+B\cos(t)}\,dt-\frac{1}{2\pi}\int_{-s/2}^{\pi-s/2}\frac{A}{1+B\cos(t)}\,dt\\
&=\frac{1}{2\pi}\int_{\pi-s/2}^{\pi+s/2}\frac{A}{1+B\cos(t)}\,dt-\frac{1}{2\pi}\int_{-s/2}^{s/2}\frac{A}{1+B\cos(t)}\,dt\\
&=\frac{1}{2\pi}\int_{-s/2}^{s/2}\frac{A}{1-B\cos(t)}\,dt-\frac{1}{2\pi}\int_{-s/2}^{s/2}\frac{A}{1+B\cos(t)}\,dt\\
&=\frac{1}{2\pi}\int_{-s/2}^{s/2}\frac{2AB\cos(t)}{1-B^2\cos^2(t)}\,dt.
\end{align*}
It is clear that this last integral is maximized over $s\in[0,2\pi]$
when $s=\pi$. Putting everything together, we deduce that,
if $b\ge a$, then
\[
c(\Omega_{a,b})=\frac{1}{2\pi}\int_{-\pi/2}^{\pi/2}
\frac{2AB\cos(t)}{1-B^2\cos^2(t)}\,dt.
\]
All that remains is to evaluate the integral.
Making the substitution $x=\sin t$, and exploiting the
fact that $A^2+B^2=1$, we have
\begin{align*}
\frac{1}{2\pi}\int_{-\pi/2}^{\pi/2}
\frac{2AB\cos(t)}{1-B^2\cos^2(t)}\,dt
&=\frac{1}{2\pi}\int_{-1}^{1} \frac{2AB}{1-B^2(1-x^2)}\,dx\\
&=\frac{1}{\pi}\int_{-1}^{1} \frac{AB}{A^2+B^2x^2}\,dx\\
&=\frac{2}{\pi}\arctan\Bigl(\frac{B}{A}\Bigr)\\
&=\frac{2}{\pi}\arctan\Bigl(\frac{1}{2}\Bigl(\frac{b}{a}-\frac{a}{b}\Bigr)\Bigr).
\end{align*}
This proves the result in the case when $b\ge a$. 
The remaining case is obtained by exchanging the roles of
$a$ and $b$.
\end{proof}

\subsection{Integral estimates}
For a general domain, the exact value of $c(\Omega)$ is often inaccessible. In this section, we will present a simple estimate which is applicable to domains $\Omega$ with a non-flat part of the boundary which leads to an upper bound on $c(\Omega)$.

Assume that we find a Borel measure $\nu$ on $\partial\Omega$
such that
\[
k^\nu_\Omega := \sup\{\|\mu_\zeta-\nu\|:\zeta\in\partial\Omega\} < 1.
\]
If so, then, for every $\phi\in C(\partial\Omega)$ with $\|\phi\|_{\partial\Omega}\le1$, we have 
\[
\Bigl|K_\Omega\phi(\zeta)-\int_{\partial\Omega}\phi\,d\nu\Bigr|\le k^\nu_\Omega,
\quad \zeta\in\partial\Omega,
\] which shows that the image of $K_\Omega \phi$ is contained in a disk of radius $k^\nu_\Omega$ centred at $\int_{\partial \Omega} \phi \, d\nu$. Thus,\[
c_\RR(\Omega) = c_\CC(\Omega) = \|K_\Omega:C(\partial\Omega)/\CC \mathbf{1} \to C(\partial\Omega)/\CC \mathbf{1}\|\le k^\nu_\Omega.
\]
One approach is to seek a positive measure $\nu$ on $\partial\Omega$ satisfying
$\mu_\zeta\ge \nu$ for all $\zeta\in\partial\Omega$. Then \[\| \mu_\zeta - \nu\| = (\mu_\zeta - \nu)(\partial \Omega) = 1 - \nu(\partial \Omega),\] and so $k^\nu_\Omega= 1- \nu(\partial\Omega)$.

We will construct the largest non-negative Borel measure $\nu$ on $\partial \Omega$ which satisfies $\mu_\zeta - \nu \geq 0$. The construction is based on the geometric interpretation of the density $\rho_\zeta(\sigma)$ in \eqref{E:DensityCircleRadiusFormula} and the quantity $R_\Omega$ appearing in \eqref{E:RSigmaDefIntro}. In order to avoid the need to establish Borel measurability of $R_\Omega$ defined as a supremum of an uncountable family as in \eqref{E:RSigmaDefIntro}, we proceed to define it in a slightly different but equivalent way. Namely, it is easy to see that, given $\sigma\in\partial\Omega$, if there exists 
a closed disk $\Delta$ such that $\Omega\subset\Delta$ and $\sigma\in\partial\Delta$,
then there exists one of smallest radius. We denote this radius by $R_\Omega(\sigma)$. Note that if $\sigma$ is not a corner of $\partial \Omega$, then the corresponding disk must be tangent to $\partial \Omega$ at $\sigma$. If no disk passing through $\sigma$ exists which contains $\Omega$, then we set $R_\Omega(\sigma):=\infty$. This happens, for instance, if $\sigma$ is contained in the interior of a line segment in $\partial \Omega$. In particular, $R_\Omega(\sigma) = \infty$ for all but a finite number of points of any polygonal domain.

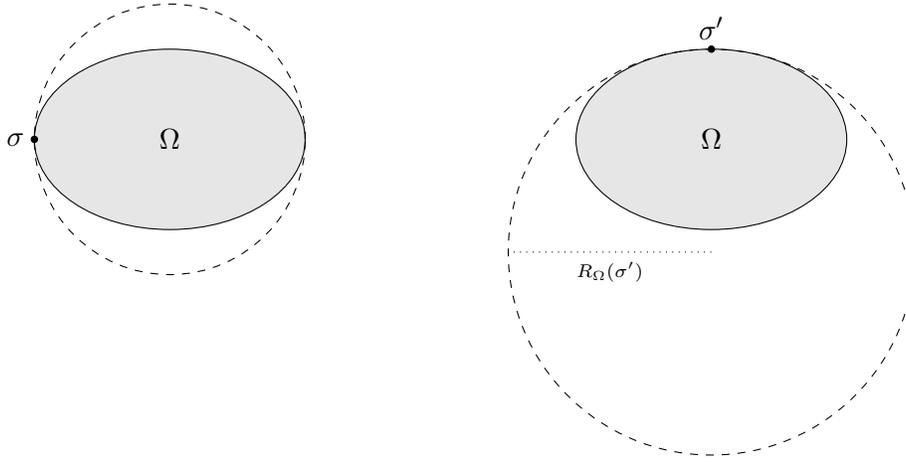
\begin{figure}
\centering
\begin{tikzpicture}[scale=0.6]
\fill[fill=gray!20] (-6,0) ellipse (3 and 2);
\draw (-6,0) node{$\Omega$};
\draw (-6,0) ellipse (3 and 2);
\draw[dashed] (-6,0) circle(3);
\filldraw (-9,0) circle (2pt) node[left]{$\sigma$};
\fill[fill=gray!20] (6,0) ellipse (3 and 2);
\draw (6,0) node{$\Omega$};
\draw (6,0) ellipse (3 and 2);
\draw[dashed] (6,-5/2) circle(9/2);
\filldraw (6,2) circle (2pt) node[above]{$\sigma'$};
\draw[ dotted] (6, -5/2) -- (3/2, -5/2) node[midway, below]{\tiny $R_\Omega(\sigma')$};
\end{tikzpicture}
\caption{A domain $\Omega$ with two circles corresponding to values $R_\Omega(\sigma)$ and $R_\Omega(\sigma')$}
\label{fig:fig6}
\end{figure}

\begin{lemma} The function
$R_\Omega:\partial\Omega\to(0,\infty]$ is lower semicontinuous.
In particular, it is Borel measurable.
\end{lemma}

\begin{proof}
Let $\sigma\in\partial\Omega$ and let $(\sigma_n)$ be a sequence in $\partial\Omega$
such that $\sigma_n\to\sigma$. We need to show that $\liminf_{n\to\infty}R_\Omega(\sigma_n)\ge R_\Omega(\sigma)$.
We can suppose that $L:=\liminf_{n\to\infty}R_\Omega(\sigma_n)<\infty$, otherwise there is nothing to prove.
Let $L'>L$. Then, replacing $(\sigma_n)$ by a subsequence, we can suppose that $R_\Omega(\sigma_n)<L'$ for all~$n$. 
Thus, for each $n$, there exists a closed disk $\Delta_n$ of radius $L'$ such that $\Omega\subset\Delta_n$ and
$\sigma_n\in\partial\Delta_n$. The sequence of centres $(c_n)$ of the disks $\Delta_n$ is bounded, so there exists a
convergent subsequence $c_{n_j}\to c$. Let $\Delta$ be the closed disk with centre $c$ and radius $L'$.
Then we have $\Omega\subset\Delta$ and $\sigma\in\partial\Delta$. It follows that $R_\Omega(\sigma)\le L'$.
As this last inequality holds for all $L'>L$, we deduce that $R_\Omega(\sigma)\le L$.
This completes the proof.
\end{proof}
We set 
\begin{equation}\label{E:InfimumMeasure}
d\nu:= \frac{ds}{2 \pi R_\Omega}.
\end{equation}
By the above lemma, $\nu$ is a non-negative Borel measure on $\partial \Omega$. For any $\zeta \in \partial \Omega$, we have 
\begin{equation} \label{E:MuZetaNuIneq}
\mu_\zeta \geq \nu.
\end{equation} To see this, note that if $\sigma$ is not a corner and $R_{\zeta,\sigma}$ is the radius of the unique circle tangent to $\partial \Omega$ at $\sigma$ and passing through $\zeta$, then $R_\Omega(\sigma) \geq R_{\zeta, \sigma}$. Therefore, according to \eqref{E:DensityCircleRadiusFormula}, 
\[
\frac{1}{2 \pi R_\Omega(\sigma)} \leq \frac{1}{2 \pi R_{\zeta, \sigma}} = \rho_\zeta(\sigma)
\] for almost every $\sigma$ with respect to arclength measure on $\partial \Omega$. Inequality \eqref{E:MuZetaNuIneq} follows. Although we shall skip a formal proof, we mention also that $\nu$ is in fact the largest measure satisfying $\mu_\zeta \geq \nu$ for all $\zeta \in \partial \Omega$. This maximality property of $\nu$ is to be interpreted in the following sense: if $\nu'$ is any measure satisfying $\mu_\zeta \geq \nu'$ for all $\zeta$, then $\nu \geq \nu'$. 

By our earlier discussion, we obtain the following upper estimate for the configuration constant:
\[
c(\Omega) \leq  1-\frac{1}{2\pi}\int_{\partial\Omega}\frac{ds}{R_\Omega}.
\]
Note that this is precisely the assertion of Theorem~\ref{T:ROmegaTheorem} stated in Section~\ref{S:IntroSec}.

We will now mention some consequences. Recall that if $\gamma$ is a plane curve of class $C^2$, then the \textit{radius of curvature} of $\gamma$ is the reciprocal of its curvature.

\begin{corollary}
If $\Omega$ has a $C^2$-boundary of length $L$, whose radius of curvature
is everywhere at most $\rho$,
then \[c(\Omega) \leq 1-\frac{L}{2\pi\rho}.\]
\end{corollary}

\begin{proof}
In this case, one sees from \eqref{E:RSigmaDefIntro} and \eqref{E:RCurvDefIntro} that $R_\Omega(\sigma)\le \rho$ for all $\sigma\in\partial\Omega$,
from which the result follows.
\end{proof}

This last result was already known. See for example 
\cite[pp.45--46]{Ga64} and \cite[pp.128--129]{KK58}.
However the proofs
in these references are quite different from the one above.

\begin{corollary}\label{T:sector}
Consider a convex circular sector \[ \Omega = \{z\in\CC:0\le |z|\le r,\, 0\le\arg(z)\le\theta\}, \] where $r>0$ and $0<\theta\le\pi$.
Then
\[
c(\Omega) \leq 1-\frac{\theta}{2\pi}.
\]
\end{corollary}

\begin{proof}
It is obvious that $R_\Omega(\sigma)=r$ for $\sigma$ in the curved
part of $\partial\Omega$, and that $R_\Omega(\sigma)=\infty$
elsewhere. Hence
\[
\frac{1}{2\pi}\int_{\partial\Omega}\frac{ds}{R_\Omega}
=\frac{1}{2\pi}\frac{r\theta}{r}=\frac{\theta}{2\pi}.
\]
The result now follows from Theorem~\ref{T:ROmegaTheorem}.
\end{proof}

\subsection{Analytic configuration constants of quadrilaterals}

Theorem~\ref{T:Maintheorem2} shows that $a(\Omega) < 1$ for every $\Omega$. Here we show by example that $a(\Omega)$ may be arbitrarily close to $1$. Figure~\ref{fig:fig5} shows a narrow quadrilateral domain for which this phenomenon occurs.

\begin{proposition}
For $\epsilon>0$, let $\Omega_\epsilon$ be the convex hull of $\{\pm1,\,\pm\epsilon i\}$. Then
\[
a(\Omega_\epsilon)\ge 1-(4/\pi)\epsilon.
\]
\end{proposition}

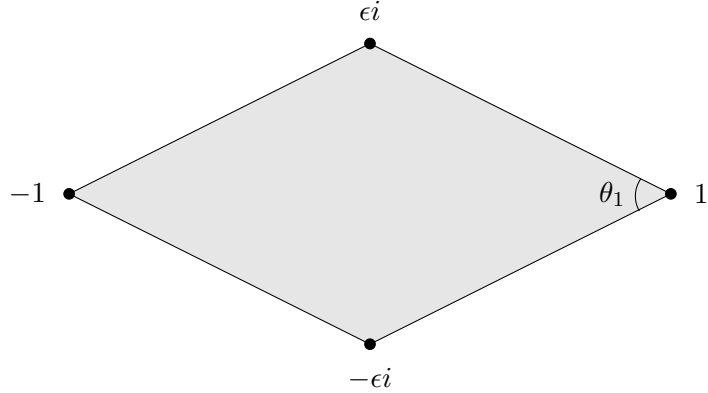
\begin{figure}
\centering
\begin{tikzpicture}[scale=4]
\draw[fill=gray!20] (-1,0)--(0,1/2)--(1,0)--(0,-1/2)--(-1,0);
\filldraw (-1,0) circle(0.5pt) node[left, xshift=-5pt]{$-1$};
\filldraw (1,0) circle(0.5pt) node[right, xshift=5pt]{$1$};
\filldraw (0,1/2) circle(0.5pt) node[above, yshift=5pt]{$\epsilon i$};
\filldraw (0,-1/2) circle(0.5pt) node[below, yshift=-5pt]{$-\epsilon i$};
\draw (0.9, 0.05) arc(145:210:0.1);
\draw (0.9,0) node[left,xshift=-2 pt]{$\theta_1$};
\end{tikzpicture}
\caption{A quadrilateral domain $\Omega_\epsilon$ with $a(\Omega_\epsilon) > 1 - (4/\pi)\epsilon$.}
\label{fig:fig5}
\end{figure}

\begin{proof}
Let $f$ be a conformal mapping of the interior of $\Omega_\epsilon$ onto the unit disk $\DD$. By Carath\'eodory's theorem, $f$ extends to a homeomorphism of $\Omega_\epsilon$ onto $\overline{\DD}$, and so clearly $f\in A(\Omega)$.
Post-composing with a suitable automorphism of $\overline{\DD}$, we may further suppose that
 $f(1)=1$ and $f(-1)=-1$. 

Consider $\zeta=1$. Recalling \eqref{E:MuZetaEq}, we have $\mu_1=(1-\theta_1/\pi)\delta_1+(\theta_1/\pi)\nu$, where $\theta_1$ is the angle of the
aperture of $\partial\Omega_\epsilon$ at $1$, and $\nu$ is a probability measure on $\partial\Omega_\epsilon\setminus\{1\}$. It follows that
\begin{align*}
\Re (K_{\Omega_\epsilon}f)(1)
&= \int_{\partial\Omega_\epsilon}(\Re f)\,d\mu_1\\
&=(1-\theta_1/\pi)\Re f(1)+(\theta_1/\pi)\int_{\partial\Omega_\epsilon\setminus\{1\}}(\Re f)\,d\nu\\
&\ge (1-\theta_1/\pi)(1)+(\theta_1/\pi)(-1)\\
&=1-2\theta_1/\pi.
\intertext{Likewise}
\Re (K_{\Omega_\epsilon}f)(-1)&\le -(1-2\theta_1/\pi).
\end{align*}
It follows that the diameter of $(K_{\Omega_\epsilon}f)(\Omega)$ is at least $2(1-2\theta_1/\pi)$,
whence $a(\Omega)\ge (1-2\theta_1/\pi)$. 

Finally, by  trigonometry, $\theta_1$ is related to $\epsilon$ by $\tan(\theta_1/2)=\epsilon$,
whence $\theta_1=2\arctan\epsilon\le2\epsilon$. The result follows.
\end{proof}

\subsection{Configuration constants equal to zero} 

Recall the estimate \eqref{E:PolFuncCalcBoundAOmega} from Section~\ref{S:IntroSec}, which will be proved in the next section. The estimate is strongest if $a(W) = 0$, and in this case we reach the conjectured bound $K = 2$. Unfortunately, the only domain $W$ for which we have $a(W) = 0$ is a disk, and in this case \eqref{E:PolFuncCalcBoundAOmega} reduces to the well-known Okubo-Ando bound from \cite{OA75}. For completeness we give a proof of the statement that $a(\Omega) = 0$ if and only if $\Omega$ is a disk. More precisely, we have the following.

\begin{proposition} Let $\Omega$ be a compact convex domain with non-empty interior. The following are equivalent:
\begin{enumerate}[(i)]
\item $\Omega$ is a disk,
\item $c(\Omega) = 0$,
\item $a(\Omega) = 0$. 
\end{enumerate}
\end{proposition}

\begin{proof}
In the case that $\Omega$ is a disk, then  \eqref{E:DensityCircleRadiusFormula} implies readily that $\rho_\zeta(\sigma)$ is a constant independent of $\zeta$, and so for every $\zeta \in \partial \Omega$, the measure $\mu_\zeta$ is a normalized arclength measure on the circular boundary $\partial \Omega$. Then it follows from the definition that $K_\Omega f$ is a constant function, and consequently $\|K_\Omega f + \CC \mathbf{1}\|_{C(\partial \Omega} = 0$, so $c(\Omega) = a(\Omega) = 0$. This shows that the implications $(i) \Rightarrow (ii)$ and $(ii) \Rightarrow (iii)$ hold.

It remains to prove $(iii) \Rightarrow (i)$. Fix a conformal mapping $\phi: \DD \to \Omega^o$, where $\DD$ is the open unit disk and $\Omega^o$ is the interior of $\Omega$. The mapping $\phi$ extends to a homeomorphism of $\partial \DD$ and $\partial \Omega$, and so it makes sense to define the probability measures $\mu^\phi_\zeta$ on $\partial \DD$ by the equation 
\[ \mu^\phi_\zeta(E) := \mu_\zeta( \phi(E) )\] where $E$ is a Borel subset of $\partial \DD$, and $\{\mu_\zeta\}_{\zeta \in \partial \Omega}$ is the double-layer potential of $\Omega$. Since $a(\Omega) = 0$, it follows that for every $f \in \A(\Omega)$ and every pair of points $\zeta, \zeta' \in \partial \Omega$ we have, by the change of variables formula, that 
\begin{align*}
0 &= \int_{\partial \Omega} f \, d\mu_\zeta - \int_{\partial \Omega} f \, d\mu_{\zeta'} \\
&= \int_{\partial \DD} f \circ \phi \, d \mu^\phi_\zeta - \int_{\partial \DD} f \circ \phi \, d \mu^\phi_{\zeta'}.
\end{align*} As $f$ varies over $\A(\Omega)$, $f \circ \phi$ varies over $\A(\DD) := \A(\overline{\DD})$, and it follows that $\mu^\phi_\zeta - \mu^\phi_{\zeta'}$ annihilates $\A(\DD)$. Then the theorem of brothers Riesz (see, for instance, \cite[Exercise 1, Chapter III]{Ga06}) implies that 
\[ \mu^\phi_\zeta - \mu^\phi_{\zeta'} = h \cdot s_{\partial \DD}\] where $h$ is a function with vanishing non-positive Fourier coefficients. Note that $h$ is real-valued, so the positive Fourier coefficients also vanish, and consequently $h \equiv 0$. Since $\zeta, \zeta'$ were arbitrary, we conclude that the hypothesis $a(\Omega) = 0$ implies that all the measures $\mu_\zeta$ are equal.

The conclusion that $\Omega$ is a disk is now a consequence of the geometric formula for $\rho_\zeta(\sigma)$ in \eqref{E:DensityCircleRadiusFormula}. Fix any $\sigma \in \partial \Omega$ which is not a corner. Since the measures $\mu_\zeta$ are all equal, so are their densities $\rho_\zeta(\sigma)$. Then the circles passing through $\zeta \in \partial \Omega$ and tangential to $\partial \Omega$ at $\sigma$ all have the same radius, and so they all coincide with each other. Thus one circle passes through all points $\zeta \in \partial \Omega$. Consequently $\partial \Omega$ is a circle, and so $\Omega$ is a disk.
\end{proof}

\section{Application to numerical ranges}
\label{S:ApplNumRan}

\subsection{Spectral constant estimate} 

Our principal motivation for the introduction of the analytic configuration constant is the following result which was mentioned in the Section~\ref{S:IntroSec} and which we will now prove.

\begin{theorem} \label{T:FunctCalcAWestimate} Let $T$ be a bounded linear operator on a Hilbert space $\HH$, and $W = \overline{W(T)}$ the closure of the numerical range of $T$. If $W$ has non-empty interior, then for every $f \in \A(W)$ we have
\[ \| f(T)\| \leq \Bigl( 1 + \sqrt{1 + a(W)} \Bigr) \| f \|_W,\]
where $a(W)$ is the analytic configuration constant in \eqref{E:AconfigConstDef}, and $\A(W)$ is the space of continuous functions on $W$ which are analytic in the interior of $W$.
\end{theorem}

Of course, if $W$ has no interior, then its convexity forces it to be a line segment. In that case $T$ is a normal operator, and the spectral theorem gives us the better estimate $\| f(T)\| \leq \|f\|_{\sigma(T)}$, where $f$ may be any Borel measurable function on the spectrum $\sigma(T)$. Thus Theorem~\ref{T:FunctCalcAWestimate} implies Theorem~\ref{T:Maintheorem3}. In what follows, we will assume that $W$ has non-empty interior.

Let us make some initial remarks before going into the proof of Theorem~\ref{T:FunctCalcAWestimate}. In the case $\sigma(T)$ is contained in the interior of $W$, then $f(T)$ is defined, as usual, through the Dunford-Riesz holomorphic functional calculus. If $\partial W \cap \sigma(T) \neq \emptyset$, then this definition does not work. Nevertheless, if $f \in \A(W)$, then it is a standard result of approximation theory that a sequence of analytic polynomials $(p_n)$ exists which converges to $f$ uniformly on $W$. In the presence of \textit{any} uniform bound of the form $\|p(T)\| \leq K \|p(T)\|_W$ for polynomials $p$, we may then define $f(T)$ as the limit of the sequence $(p_n(T))$ in the operator norm. Such bounds are known to exists, the strongest known bound $K \leq 1 + \sqrt{2}$ being due to Crouzeix and Palencia. Theorem~\ref{T:FunctCalcAWestimate} improves this estimate given information about the numerical range of $T$. 

Our proof of Theorem~\ref{T:FunctCalcAWestimate} combines the argument of Crouzeix and Palencia from \cite{CP17} with ideas of Schwenninger and de Vries from \cite{SdV23}, where bounds for various functional calculi are derived as a consequence of the existence of \textit{extremal functions} and \textit{extremal vectors}. Let $U$ be an open set in the plane, and $H^\infty(U)$ be the algebra of bounded holomorphic functions on $U$. Given an operator $T: \HH \to \HH$ with $\sigma(T)$ contained in $U$, it is elementary that the quantity
\[ 
\sup \Big\{ \|f(T)\| : f \in H^\infty(U), \|f\|_U \leq 1\Big\}
\]
is finite. A normal-families argument shows that an $f \in H^\infty(U)$ exists with $\|f\|_U = 1$ for which the supremum above is attained. Any such $f$ will be called for an \textit{extremal function}. If, moreover, a vector $x \in \HH$ with $\|x\|_\HH = 1$ exists for which 
\[ 
\sup \Big\{ \|f(T)\| : f \in H^\infty(U), \|f\|_U \leq 1\Big\} = \|f(T)x\|_\HH
\]
then we will say that $x$ is an \textit{extremal vector}, and $(f,x)$ is an \textit{extremal pair}. Unless $\dim \HH < \infty$, an extremal vector might not exist, but we will be able to reduce the proof to the finite-dimensional case. The importance of the concept of extremal pairs $(f,x)$ stems from the following result. We refer the reader to \cite[Theorem 4.5]{BGGRSW20} for a proof (see also \cite[Proposition 3]{SdV23}).

\begin{lemma}\label{L:extremal}
Let $T:\HH \to \HH$ be a bounded linear operator, and $U$ be an open neighbourhood of $\sigma(T)$. Let $(f,x)$ be a corresponding extremal pair. If $\|f(T)\|>1$, then $f(T)x$ is orthogonal to $x$ in $\HH$:
\[
\langle f(T)x,x\rangle_\HH =0.
\]
\end{lemma}

The next two lemmas will reduce our task to consideration of finite-dimensional Hilbert spaces, in which extremal vectors exist, and will dispose of the problematic set $\sigma(T) \cap \partial W$. The first observation is essentially contained in \cite[Proposition 9]{SdV23}.

\begin{lemma} \label{L:AWEstimateLemma1}
Let $\Omega$ be a compact convex domain with non-empty interior. If for some $K > 0$ the estimate 
\[ \|p(T)\| \leq K \| p \|_\Omega\] holds for every polynomial $p$ and every operator $T$ on a finite-dimensional Hilbert space with $W(T)$ contained in the interior of $\Omega$, then the same estimate with the same constant $K$ holds also for operators $T$ on infinite-dimensional Hilbert spaces with $\overline{W(T)}$ contained in the interior of $\Omega$.
\end{lemma}

\begin{proof}
Let $T: \HH \to \HH$ be as above, with $\dim \HH = \infty$. It suffices to show that \[ \|p(T)x\|_\HH \leq K \| p \|_\Omega \|x\|_\HH\] holds for every analytic polynomial $p$ and every $x \in \HH$. Note that $p(T)x$ is contained in the finite-dimensional subspace $\mathcal{K}$ spanned by $\{x, Tx, \ldots, T^dx\}$, where $d$ is the degree of the polynomial $p$. If $\Pi: \HH \to \mathcal{K}$ is the orthogonal projection, then $p(T)x = \Pi p(T)x = p(\Pi T)x$, where $\Pi T : \mathcal{K} \to \mathcal{K}$ is an operator on a finite-dimensional Hilbert space. Since $W(\Pi T) \subset \overline{W(T)}$, our hypothesis implies
\[ \|p(T)x\|_\HH = \|p(\Pi T)x\|_\mathcal{K} \leq K \|p\|_\Omega \|x\|_\mathcal{K} = K \|p\|_\Omega \|x\|_\HH.\]
The lemma follows.
\end{proof}

The proof of the next lemma will use affine invariance of the configuration constants. Let us fix $\alpha, \beta \in \CC$, $\alpha \neq 0$, and an affine mapping $A(z) := \alpha z + \beta$. Then $A$ is a conformal transformation of $\CC$ with the additional property of taking a line segment of length $L$ to a line segment of length $|\alpha|L$, and a circle of radius $R$ to a circle of radius $|\alpha|R$. Let $\widetilde{\Omega} = A(\Omega)$ be the affine image of $\Omega$ under $A$, and recall the formula for the Neumann-Poincaré kernel in \eqref{E:MuZetaEq} and its geometric interpretation. If $\zeta, \sigma \in \partial \Omega$, $E$ is a Borel subset of $\partial \Omega$, and $s, \widetilde{s}$ are the arclength measures on $\partial \Omega$ and $\partial \widetilde{\Omega}$ respectively, then it follows from the properties of $A$ listed above that
\begin{enumerate}[(i)]
\item $\theta_\zeta = \theta_{A(\zeta)}$,
\item $|\alpha| s(E) = \widetilde{s}(A(E))$,
\item $|\alpha| R_{\zeta, \sigma} = R_{A(\zeta), A(\sigma)}$. 
\end{enumerate}
A consequence is that the Neumann-Poincaré kernels $\{\mu_\zeta\}_{\zeta \in \partial \Omega}$ and $\{\widetilde{\mu}_{A(\zeta)}\}_{A(\zeta) \in \partial \Omega}$ of the respective domains satisfy 
\[ \widetilde{\mu}_{A(\zeta)}(A(E)) = \mu_\zeta(E), \quad E \text{ a Borel subset of } \partial \Omega.\] Then a change of variables shows that $K_\Omega (\widetilde{f} \circ A) = K_{\widetilde{\Omega}} \widetilde{f}$ for any $\widetilde{f} \in C(\partial \widetilde{\Omega})$, and it follows that 
\[ a(\Omega) = a(\widetilde{\Omega}), \quad c(\Omega) = c(\widetilde{\Omega}).\] Armed with these equalities, we make our second observation.

\begin{lemma} \label{L:AWEstimateLemma2}
Assume that the estimate 
\begin{equation}\label{E:aOmegaEst} \|p(T)\| \leq \Bigl( 1 + \sqrt{1 + a(\Omega)} \Bigr) \|p \|_\Omega
\end{equation} holds for every polynomial, every compact convex domain $\Omega$, and every operator $T$ for which $W(T)$ is contained in the interior of $\Omega$. Then Theorem~\ref{T:FunctCalcAWestimate} holds.
\end{lemma}

\begin{proof} Replacing $T$ by an operator $T + \beta I$ for some $\beta \in \CC$, we may assume that $0$ lies in the interior of $W(T)$. Let $W = \overline{W(T)}$, and 
\[
W_r = \{ rz : z \in W \}, \quad r > 1.
\] Then $W_r$ is a convex domain which contains $W$ in its interior. By our assumption, for any analytic polynomial $p$ we have \[ \|p(T)\| \leq \Bigl( 1 + \sqrt{1 + a(W_r)} \Bigr) \|p\|_{W_r}.\] Since $W_r$ is an affine image of $W$, we have $a(W_r) = a(W)$. Since this holds for all $r > 1$, and since $\lim_{r \to 1} \|p\|_{W_r} = \|p\|_W$, we may let $r \to 1$ to obtain the desired estimate whenever $p$ is an analytic polynomial. The estimate for $f \in \A(W)$ follows by density of polynomials in $\A(W)$.
\end{proof}

\begin{proof}[Proof of Theorem~\ref{T:FunctCalcAWestimate}]
By Lemma~\ref{L:AWEstimateLemma2}, it will be sufficient to establish the estimate \eqref{E:aOmegaEst} whenever $\Omega$ contains $W(T)$ in its interior $\Omega^o$. Moreover, by Lemma~\ref{L:AWEstimateLemma1}, we may assume that $T$ is an operator on a finite-dimensional Hilbert space $\HH$. Let $U = \Omega^o$ and $(f,x)$ be an extremal pair corresponding to $U$ and the operator $T$. If $\|f(T)\| \leq 1$, then \eqref{E:aOmegaEst} certainly holds, so we may assume that $\|f(T)\| > 1$. 

Let $(f_n)$ be a sequence in $A(\Omega)$ such that $\|f_n\|_\Omega\le 1$
and $f_n\to f$ locally uniformly in $\Omega$.
Then $f_n(T)\to f(T)$ in operator norm.
Set $g_n:=K_{\Omega}\overline{f}_n$. It is shown in \cite[Lemmas 2.1 and 2.3]{CP17} that $g_n \in \A(\Omega)$ and 
\begin{equation}\label{E:2bound}
\|f_n(T)+g_n(T)^*\|\le 2.
\end{equation}
For each $n$, we may choose $\lambda_n\in\CC$ such that 
\[
\|g_n+\lambda_n \mathbf{1}\|_\Omega=\inf_{\lambda\in\CC}\|g_n+\lambda \mathbf{1}\|_\Omega  \le a(\Omega).
\]
We now have the following identity:
\begin{align}\label{E:longeqn}
\langle f_n(T)x,\, f_n(T)x\rangle_\HH
=& \, \langle f_n(T)x, (f_n(T)+g_n(T)^*)x\rangle_\HH \\
&-\langle f_n(T)x, (g_n+\lambda_n \mathbf{1})(T)^*x\rangle_\HH \nonumber \\
&+\lambda_n\langle f_n(T)x,x\rangle_\HH. \nonumber
\end{align}
Let us consider each of the terms in this identity. By the choice of $x$, we have
\[
\langle f_n(T)x,f_n(T)x\rangle_\HH =\|f_n(T)x\|^2\underset{n\to\infty}\longrightarrow\|f(T)x\|^2=\|f(T)\|^2.
\]
Also, from \eqref{E:2bound} and the Cauchy--Schwarz inequality,
\[
\bigl|\langle f_n(T)x, (f_n(T)+g_n(T)^*)x\rangle_\HH \bigr|\le 2\|f_n(T)\|\underset{n\to\infty}\longrightarrow2\|f(T)\|.
\]
By Lemma~\ref{L:extremal}, we have
\begin{align*}
\bigl|\langle f_n(T)x, (g_n+\lambda_n \mathbf{1})(T)^*x\rangle\bigr|
&=
\bigl|\langle (f_n(g_n+\lambda_n \mathbf{1}))(T)x, x\rangle\bigr| \\
&\le \|f_n(g_n+\lambda_n \mathbf{1})\|_\Omega \\
&\le\|g_n+\lambda_n \mathbf{1}\|_\Omega \\
&\le a(\Omega).
\end{align*}
By Lemma~\ref{L:extremal} again, $\langle f(T)x,x\rangle_\HH =0$.
Since the sequence $(\lambda_n)$ is certainly
bounded (indeed $|\lambda_n|\le 2$), we deduce that
\[
\lambda_n\langle f_n(T)x,x\rangle_\HH \underset{n\to\infty}\longrightarrow0.
\]
Thus, letting $n\to\infty$ in \eqref{E:longeqn}, we deduce that
\[
\|f(T)\|^2\le 2\|f(T)\|+a(\Omega).
\]
Hence
\[
\|f(T)\|\le 1+\sqrt{1+a(\Omega)}.
\] 
In particular, for every polynomial $p$ with $\|p\|_\Omega = 1$ we have 
\[ \|p(T)\| \leq \|f(T)\| \leq 1+\sqrt{1+a(\Omega)},\] since $f$ is extremal. This is equivalent to \eqref{E:aOmegaEst}, and so the proof is complete.
\end{proof}

\appendix

\section{Double-layer potential on a general convex domain}
\label{sec:appendixA}

\subsection{Convex domains} Let $\Omega$ be a compact convex domain in the plane $\mathbb{C}$ with non-empty interior $\Omega^o$. We will be making no assumptions regarding smoothness of the boundary $\partial \Omega$. However, convexity itself implies that $\partial \Omega$ is a rectifiable simple closed curve with some additional properties. 

The orientation of $\partial \Omega$ is to be counter-clockwise (that is, positive), and we use $\sigma' \uparrow \sigma$ and $\sigma' \downarrow \sigma$ to denote, respectively, the counter-clockwise and clockwise one-sided convergence of $\sigma'$ to $\sigma$ within $\partial \Omega$. As a consequence of convexity of $\Omega$, the one-sided tangent angles exist at every point $\sigma \in \partial \Omega$, are locally given by 
\[\alpha_+(\sigma) := \lim_{\sigma' \downarrow \sigma} \arg ( \sigma' - \sigma), \quad \alpha_-(\sigma) := \lim_{\sigma' \uparrow \sigma} \arg (\sigma - \sigma'),\] and satisfy \[\alpha_-(\sigma) \leq \alpha_+(\sigma).\] Strict inequality may occur at most at a countable subset of $\partial \Omega$. If it occurs at $\sigma$, then we say that $\partial \Omega$ has a \textit{corner} at $\sigma$. At any point which is not a corner, the tangent angle \[\alpha(\sigma) := \alpha_+(\sigma) = \alpha_-(\sigma)\] is well-defined, and so is the tangent $T(\sigma) := e^{i \alpha(\sigma)}$ itself. If $t \mapsto \gamma(t)$ is any (positively-oriented) parametrization of $\partial \Omega$, and we set $\alpha(\sigma) = \alpha_+(\sigma)$ at the corners, then the locally defined function $\alpha(\gamma(t))$ is increasing in $t$, and consequently the tangent $T$ is continuous at every point which is not a corner of $\partial \Omega$. At a corner, the discontinuity of $T$ amounts to a jump of the argument of $T$. We denote by $N(\sigma) := -i T(\sigma)$ the outward-pointing normal at $\sigma \in \partial \Omega$. 

\subsection{Double-layer potential} Let $\Omega^o$ denote the interior of $\Omega$. To each $z \in \Omega^o$ we associate the measure $\mu_z$ on $\partial \Omega$, which for any arc $J \subset \partial \Omega$ satisfies
\begin{equation}
\label{E:MuZDefAppendix}
\mu_z(J) = \frac{1}{\pi}\int_J d\arg (\sigma - z) = \frac{1}{\pi}\bigl(\text{angle subtended at $z$ by $J$}\bigr).
\end{equation}

Here $\arg(\sigma - z)$ is any locally defined continuous determination of the argument function on $\partial \Omega$. Non-negativity of $\mu_z$ follows from convexity of $\Omega$ and our choice of positive orientation of $\partial \Omega$. With respect to this orientation, every arc $J =(a,b) \subset \partial \Omega$ has a start-point $a$ and an end-point $b$, and it is easy to see that
\[
\mu_z(J) = \frac{\arg(b - z) - \arg(a - z)}{\pi}.
\] In particular, $\mu_z(\partial \Omega) = 2$. 

The measure $\mu_z$ is absolutely continuous with respect to arclength $s$ on $\partial \Omega$. Indeed, if $\sigma_0 \in \partial \Omega$, $J_n = (a_n, b_n)$ is a sequence of arcs of $\partial \Omega$ which are shrinking to $\sigma_0$, and $|J_n|$ are the corresponding arclengths, then 
\begin{align*}
\frac{\pi \mu_z(J_n)}{|J_n|} &= \frac{1}{|J_n|}\int_{J_n} d\arg(\sigma - z) \\ &= \frac{\arg(b_n - z) - \arg(a_n - z)}{|J_n|} \\ 
&= \Im \Bigg( \frac{\log ( b_n - z) - \log (a_n - z)}{b_n-a_n} \cdot \frac{b_n - a_n}{|J_n|}\Bigg).
\end{align*} 
We use above an appropriate locally defined holomorphic branch of the logarithm. As $n \to \infty$, the first factor inside the brackets satisfies \[ \lim_{n \to \infty} \frac{\log ( b_n - z) - \log (a_n - z)}{b_n-a_n} = \frac{1}{\sigma_0 - z},\] while the second factor stays bounded as a consequence of the inequality $|b_n - a_n| \leq |J_n|$. Thus 
\[ \limsup_{n \to \infty} \frac{\mu_z(J_n)}{|J_n|} < \infty
\] and from elementary measure theory we obtain that $\mu_z$ is absolutely continuous with respect to $s$. If moreover $\sigma_0$ is not a corner, then it can be shown that
\[
\lim_{n \to \infty} \frac{|J_n|}{|b_n-a_n|} = 1,
\]
and so in additional to boundedness we even have the convergence 
\[ \lim_{n \to \infty} \frac{b_n - a_n}{|J_n|} = \lim_{n \to \infty} \frac{b_n - a_n}{|b_n - a_n|} = T(\sigma_0) = i N(\sigma_0).\]
Thus the Radon-Nikodym derivative satisfies
\begin{equation} \label{E:MuDensityAppendix}
\rho_z(\sigma) := \frac{d \mu_z }{ds}(\sigma) =  \frac{1}{\pi}\Im \Bigg( \frac{T(\sigma)}{\sigma - z}\Bigg) = \frac{1}{\pi}\Re \Bigg ( \frac{N(\sigma)}{\sigma - z} \Bigg)
\end{equation}
at every $\sigma \in \partial \Omega$ which is not a corner. 

\subsection{Boundary kernel} The Neumann-Poincaré kernel is the boundary version of the family of measures $\{\mu_z\}_{z \in \Omega^o}$ introduced above. To each point $\zeta \in \partial \Omega$ we associate the Borel probability measure on $\partial \Omega$ defined by \eqref{E:MuZDefAppendix} for arcs $J \subset \partial \Omega$ not containing the point $\zeta$. Because $\zeta \in \partial \Omega$, this definition implies that $\mu_\zeta( \partial \Omega \setminus \{\zeta\}) = \theta_\zeta/\pi$, where $\theta_\zeta = \pi - \alpha_+(\zeta) + \alpha_-(\zeta)$ can be interpreted as the angle of the aperture at $\zeta$. Indeed, $\theta_\zeta$ is equal to the increase in the argument of $\sigma - \zeta$ as we traverse one loop around $\partial \Omega$ starting and ending at the point $\zeta$, and since $\mu_\zeta$ is a probability measure, we must have \[ \mu_\zeta(\{\zeta\}) = 1 - \frac{\theta_\zeta}{\pi}.\]
With the exception of this possible point mass, $\mu_\zeta$ is otherwise absolutely continuous with respect to arclength. The corresponding Radon-Nikodym derivative is given by

\begin{equation} \label{E:MuZetaDensityAppendix}
\rho_\zeta(\sigma) := \frac{d \mu_\zeta }{ds}(\sigma) =  \frac{1}{\pi}\Im \Bigg( \frac{T(\sigma)}{\sigma - \zeta}\Bigg) = \frac{1}{\pi}\Re \Bigg ( \frac{N(\sigma)}{\sigma - \zeta} \Bigg).
\end{equation} The formula \eqref{E:MuZetaDensityAppendix} is established analogously to \eqref{E:MuDensityAppendix}.
All in all, the measure $\mu_\zeta$ can be decomposed as 
\[ d\mu_\zeta = (1- \theta_\zeta/\pi) d \delta_\zeta + \rho_\zeta d s,
\]
where $\delta_\zeta$ is a unit mass at $\zeta \in \partial \Omega$, $\theta_\zeta$ is the angle of the aperture at $\zeta$ (with the convention that $\theta_\zeta = \pi$ if $\zeta$ is not a corner), and where the density $\rho_\zeta$ is given by \eqref{E:MuZetaDensityAppendix}.

\subsection{Weak-star convergence} We establish now that 
\[ \lim_{z \to \zeta} \mu_z = \delta_\zeta + \mu_\zeta \] in the sense of the weak-star topology on measures. Note that if $B = B(\zeta, \delta)$ is a ball of radius $\delta > 0$ centred at $\zeta \in \partial \Omega$, then expressions \eqref{E:MuDensityAppendix} and \eqref{E:MuZetaDensityAppendix} for the densities of $\mu_z$ and $\mu_\zeta$ show that
\begin{equation}
\label{E:MUzMuZetaBCompConv}
\lim_{z \to \zeta} \int_{\partial \Omega \setminus B} f \, d\mu_z =  \int_{\partial \Omega \setminus B} f \, d\mu_\zeta
\end{equation} for every $f \in C(\partial \Omega)$. In particular, choosing $f = \mathbf{1}$, we obtain
\[ 2 =  \mu_\zeta(\partial \Omega \setminus B) + \lim_{z \to \zeta} \mu_z(B).\] Since \[ \lim_{\delta \to 0} \mu_\zeta(\partial \Omega \setminus B) = \mu_\zeta(\partial \Omega \setminus \{\zeta\}) = \theta_\zeta/\pi\] we see that given $\epsilon > 0$ for all sufficiently small $\delta > 0$ we will have
\[ \limsup_{z \to \zeta} \, |\mu_z(B) - 2 + \theta_\zeta/\pi| \leq \epsilon.\] 
Returning to general $f \in C(\partial \Omega)$, we have
\begin{align*}
\int_{\partial \Omega} f \, d\mu_z - \int_{\partial \Omega} f \, d[\delta_\zeta + \mu_\zeta] =& \int_{\partial \Omega \setminus B} f \, d\mu_z  - \int_{\partial \Omega \setminus B} f \, d\mu_\zeta \\
&+ \int_{B} \big(f - f(\zeta)\big) \, d\mu_z \\
&+ f(\zeta) \big(\mu_z(B) - 2  + \theta_\zeta/\pi \big) \\
&- \int_{B \setminus \{\zeta\}} f \, d\mu_\zeta.
\end{align*}
On the right-hand side, the first term tends to zero as $z \to \zeta$, the second can be made arbitrarily small by continuity of $f$, the crude estimate $\mu_z(B) \leq 2$ and choice of sufficiently small $\delta$, the third is dominated in modulus by $\|f\|_{\partial \Omega} \cdot \epsilon$ for $z$ sufficiently close to $\zeta$, and the fourth is dominated by $\|f\|_{\partial \Omega} \cdot \mu_\zeta(B \setminus \{\zeta\})$ which also can be made arbitrarily small by choice of sufficiently small $\delta$. The desired weak-star convergence follows.

\section*{Acknowledgement}

We are grateful to the referee for the careful reading of the paper, and for the suggestions which we used to improve the manuscript.

%\section*{Declarations}
%
%\subsection*{Data availability statement} Data sharing not applicable to this article as no datasets were generated or analysed during
%the current study.
%
%\subsection*{Conflict of interest} The corresponding author states that there is no conflict of interest.

\bibliographystyle{plain}
\bibliography{mybib}

\end{document}